\documentclass[reqno,12pt,letterpaper]{amsart}
\usepackage[proof]{sdmacros}

\title[Pollicott--Ruelle resolvent and Sobolev regularity]{Pollicott--Ruelle resolvent\\ and Sobolev regularity}
\author{Semyon Dyatlov}
\email{dyatlov@math.mit.edu}
\address{Department of Mathematics, Massachusetts Institute of Technology, Cambridge, MA 02139}

\dedicatory{In memory of Louis Nirenberg}

\begin{document}

\begin{abstract}
In this note we compute the threshold regularity
for meromorphic continuation of the Pollicott--Ruelle resolvent
of an Anosov flow as an operator on anisotropic Sobolev spaces,
in the setting of lifts to general vector bundles.
These thresholds are related to the Sobolev regularity needed for
the decay of correlations.
\end{abstract}

\maketitle

\addtocounter{section}{1}
\addcontentsline{toc}{section}{1. Introduction}

Let $M$ be a compact $d$-dimensional $C^\infty$ manifold (without boundary)
and $X\in C^\infty(M;TM)$ be a nonvanishing vector field.
This field generates a flow
$$
\varphi^t:=\exp(tX):M\to M,\quad
t\in\mathbb R.
$$
A fundamental topic in dynamical systems is the study of the behavior
of \emph{correlations}
$$
\rho_{f,g}(t):=\int_M (f\circ\varphi^{-t})g,\quad
t\in\mathbb R
$$
where $f$ is an $L^2$ function and $g$ is an $L^2$ density on $M$.
(Here the line bundle of densities on $M$ is used because
the integral of a density over $M$ is invariantly defined;
we do not assume a priori that the flow preserves a smooth volume form.)
In particular, one is interested in \emph{mixing} (when $\rho_{f,g}(t)$
has a limit as $t\to\infty$) and also in the stronger property
of \emph{exponential mixing} (when the remainder in the mixing property
decays exponentially fast as $t\to\infty$). We note that even when exponential mixing
is known, it does not hold for all $f,g\in L^2$, instead one has to restrict
to more regular functions.

In this paper we focus on the case when $\varphi^t$ is an \emph{Anosov flow},
that is the tangent spaces to $M$ decompose into the flow, stable, and unstable directions~--
see~\S\ref{s:anosov-def} for a precise definition. There are many examples
of such flows, including geodesic flows on manifolds of negative curvature
(see~\S\ref{s:examples-geodesic}).
A key tool in studying long time asymptotics of correlations
is the \emph{Pollicott--Ruelle resolvent}
$$
R_X(\lambda)f=\int_0^\infty e^{-\lambda t}(f\circ\varphi^{-t})\,dt,\quad
\Re\lambda>0,\quad f\in C^\infty(M)
$$
with the integral converging in the space of continuous functions $C^0(M)$.
The integral $\int_M (R_X(\lambda)f)g$ is the Fourier--Laplace transform
of the correlation $\rho_{f,g}(t)$ at~$\lambda$. 

Since $X$ is a smooth vector field, differentiation along it
defines a first order differential operator
which we also denote by~$X$. 
This operator acts in particular on the space of smooth functions $C^\infty(M)$
and on the space of distributions $\mathcal D'(M)$.
Now, $R_X(\lambda)$ is an inverse of $X+\lambda$ in the following
sense:
\begin{equation}
  \label{e:inversor}
R_X(\lambda)(X+\lambda)f=(X+\lambda)R_X(\lambda)f=f\quad\text{for all}\quad
f\in C^\infty(M),\quad
\Re\lambda>0.
\end{equation}
A fundamental property of $R_X(\lambda)$ is that it continues
meromorphically to the entire complex plane:
\begin{theo}
  \label{t:old}
Assume that $X$ is an Anosov flow. Then $R_X(\lambda)$ admits a meromorphic extension
$$
R_X(\lambda):C^\infty(M)\to\mathcal D'(M),\quad
\lambda\in\mathbb C.
$$
\end{theo}
The poles of the extended family $R_X(\lambda)$, called the \emph{Pollicott--Ruelle resonances} of~$\varphi^t$, are the complex characteristic
frequencies governing the decay of correlations.
They also appear as singularities (zeroes and/or poles) of dynamical zeta functions.

A typical proof of Theorem~\ref{t:old} is to use~\eqref{e:inversor}
and construct the meromorphic continuation of $R_X(\lambda)$
as the inverse of $X+\lambda$ acting between two Banach spaces of distributions
$D\to H$ which are carefully designed so that $X+\lambda:D\to H$ is a Fredholm operator.
This gives the continuation to a half-plane $\Re\lambda> -c$ where
the value of the constant $c$ depends on the choice of the spaces,
and it is possible to choose $D,H$ to make $c$ arbitrarily large.

The present paper establishes a version of Theorem~\ref{t:old}
in the more general setting of a smooth vector bundle $\mathcal E$ over $M$
and an arbitrary lift $\mathbf X:C^\infty(M;\mathcal E)\to C^\infty(M;\mathcal E)$ of~$X$~--
see~\S\ref{s:more-general-operators} for details
and~\S\ref{s:examples-1} for examples.
It is already known that such an extension holds, however in this paper we compute
the needed regularity for the spaces on which Fredholm property holds.
This can be used in particular to better understand the regularity assumptions
for exponential decay of correlations as well as regularity of resonant states.

We use an anisotropic Sobolev space $H^{\mathfrak m}(M;\mathcal E)$ associated to 
a weight function $\mathfrak m\in C^\infty(T^*M\setminus 0;\mathbb R)$ which
is homogeneous of degree~0. This function needs to satisfy natural dynamical
assumptions (see~\S\ref{s:anisotropic-spaces}), in particular to it correspond
two numbers
$$
m_u\leq 0\leq m_s
$$
such that $H^{m_s}(M;\mathcal E)\subset H^{\mathfrak m}(M;\mathcal E)\subset H^{m_u}(M;\mathcal E)$. See Adam--Baladi~\cite[\S3.3]{Adam-Baladi} for the threshold regularity computation for the case of trivial one-dimensional bundles, giving~\eqref{e:threshold-total} in that case (see also Guillarmou--Poyferr\'e--Bon\-thonneau~\cite[Appendix~A]{Guillarmou-Poyferre}), Wang~\cite{Wang-Besov} for radial estimates
giving regularity in the more general Besov spaces in the scalar case,
and Bonthonneau--Lefeuvre~\cite{Bonthonneau-Lefeuvre} for a related result
giving the regularity threshold
in the case of general bundles for H\"older--Zygmund spaces.
For an estimate of the regularity threshold in anisotropic Banach spaces
in the related case of Anosov maps, see~\cite[Theorem~4.1]{Baladi-Tsujii-Spectra} or~\cite[Theorem~6.12]{Baladi-Book}.

The main result of this paper, Theorem~\ref{t:main} in~\S\ref{s:main},
shows meromorphic continuation of the Pollicott--Ruelle resolvent $R_{\mathbf X}(\lambda)$
associated to the lift $\mathbf X$ to a half-plane which is explicitly described
in terms of $m_u,m_s$ and the dynamics of the flow $\varphi^t$. More precisely,
the condition on $\lambda$ is that there exists $\varepsilon>0$ and a constant $C$
such that for all $x\in M$ and $t\geq 0$
\begin{equation}
\label{e:threshold-total}
\begin{aligned}
|\det d\varphi^t(x)|^{1\over 2}\cdot\|\mathscr T_{\mathbf X}^t(x)\|
\cdot \big\|d\varphi^t(x)|_{E_s}\big\|^{-m_u}&\leq C e^{(\Re\lambda-\varepsilon)t},\\
|\det d\varphi^t(x)|^{1\over 2}\cdot\|\mathscr T_{\mathbf X}^t(x)\|
\cdot \big\|d\varphi^t(x)^{-1}|_{E_u}\big\|^{m_s}&\leq C e^{(\Re\lambda-\varepsilon)t}.
\end{aligned}
\end{equation}
Here $E_u,E_s$ are the unstable/stable spaces of the flow
and $\mathscr T_{\mathbf X}^t(x):\mathcal E(x)\to\mathcal E(\varphi^t(x))$
is the parallel transport associated to the lift~$\mathbf X$.
See~\S\ref{s:threshold} and the statement of Theorem~\ref{t:main} for details
and~\S\ref{s:examples-2} for examples.

The use of anisotropic H\"older and Sobolev spaces to prove Theorem~\ref{t:old}
and an analogous statement in the related setting of Anosov maps
has a long tradition, see in particular the works of
Blank--Keller--Liverani~\cite{Blank-Keller-Liverani},
Liverani~\cite{Liverani-contact,Liverani-Fredholm},
Gouezel--Liverani~\cite{Gouezel-Liverani},
Baladi--Tsujii~\cite{Baladi-Tsujii}, and
Butterley--Liverani~\cite{Butterley-Liverani}. We use the \emph{microlocal approach} originating in the papers of Faure--Roy--Sj\"ostrand~\cite{Faure-Roy-Sjostrand} and Faure--Sj\"ostrand~\cite{Faure-Sjostrand}.
See the review of Zworski~\cite[\S4]{Zworski-resonances-review} for a comprehensive introduction
to this microlocal approach.
Our proof is similar in structure to the one in the paper of Dyatlov--Zworski~\cite{DZ-zeta}
on dynamical zeta functions. (See also the work of Dyatlov--Guillarmou~\cite{DG-open,DG-afterword} for the more
general setting of basic sets of Axiom A flows.)
The main difference between the present paper and~\cite{DZ-zeta}
is the precise analysis of what regularity is needed for radial estimates~--
see~\S\S\ref{s:adj-comm}, \ref{s:threshold}, and~\ref{s:radial}.

We also address a minor mistake present in~\cite{DZ-zeta,DG-open}:
when the vector bundle $\mathcal E$ is not trivial, it is not possible to extend pseudodifferential operators on $C^\infty(M)$ canonically to operators
on $C^\infty(M;\mathcal E)$. Thus all the pseudodifferential cutoffs $A,B,B_1,\dots$
used in the propagation estimates in~\cite{DZ-zeta,DG-open} should be taken to
be principally scalar operators rather than operators on $C^\infty(M)$.

For applications of anisotropic spaces to exponential mixing for contact flows,
see the works of Liverani~\cite{Liverani-contact}, Tsujii~\cite{Tsujii-contact},
and Nonnenmacher--Zworski~\cite{Nonnenmacher-Zworski-Inventiones}. We note that the latter paper~\cite{Nonnenmacher-Zworski-Inventiones}
uses the microlocal approach and thus could be potentially combined with the present
result to yield exponential mixing for more general bundles, however
in the case when $\mathbf X^*\neq -\mathbf X$ more adjustments would be needed to the argument there.

\section{Anosov flows}

\subsection{Definition}
\label{s:anosov-def}

As in the introduction, we assume that $X$ is a nonvanishing vector
field on a compact manifold $M$ and $\varphi^t=\exp(tX)$
is the corresponding flow.
\begin{defi}
We say that $\varphi^t$ is an \emph{Anosov flow} if
there exists a splitting of tangent spaces into
the \emph{flow/unstable/stable spaces}
\begin{equation}
  \label{e:stun-original}
T_xM=E_0(x)\oplus E_u(x)\oplus E_s(x),\qquad
x\in M
\end{equation}
such that:
\begin{itemize}
\item $E_0(x)=\mathbb RX(x)$;
\item $E_u(x),E_s(x)$ depend continuously on~$x$ and are
invariant under the flow:
$$
d\varphi^t(x)E_u(x)=E_u(\varphi^t(x)),\quad
d\varphi^t(x)E_s(x)=E_s(\varphi^t(x));
$$
\item we have the exponential contraction property under the differential of the flow,
\begin{equation}
  \label{e:Anosov-contraction}
|d\varphi^t(x)v|\leq Ce^{-\theta|t|}|v|\quad\text{if}\quad
\begin{cases}
v\in E_u(x),& t\leq 0\quad\text{or}\\
v\in E_s(x),& t\geq 0.
\end{cases}
\end{equation}
Here $C,\theta>0$ are some constants and we fix an arbitrary
Riemannian metric on~$M$; $C$ depends on the choice of the metric but $\theta$
does not.
\end{itemize}
\end{defi}
\Remark The dependence of $E_u(x),E_s(x)$ on the base point~$x$
is H\"older continuous but typically not $C^\infty$,
see for example~\cite{Hurder-Katok}. 

In this paper we always assume that $\varphi^t$ is an Anosov flow.
It is sometimes useful to make additional assumptions, given by
\begin{defi}
\label{d:contact-def}
Let $X$ be a nonvanishing vector field on a manifold~$M$. We say that the flow
$\varphi^t=\exp(tX)$ is:
\begin{itemize}
\item a \emph{volume preserving} flow, if there exists a $C^\infty$ density $\mu$ on $M$
which is invariant under pullback by $\varphi^t$;
\item a \emph{contact} flow, if $d=\dim M$ is odd and there exists a 1-form $\alpha\in C^\infty(M;T^*M)$ such that $\alpha\wedge (d\alpha)^{d-1\over 2}$ is nonvanishing,
$\iota_X \alpha=1$, and $\iota_X d\alpha=0$.
\end{itemize}
\end{defi}
\Remark For contact flows, the form $\alpha$ is called a \emph{contact form} and
$X$ is called the \emph{Reeb vector field} associated to $\alpha$.
The manifold $M$ is oriented by requiring that
$d\vol_\alpha:=\alpha\wedge (d\alpha)^{d-1\over 2}$
be positive. Moreover, $d\vol_\alpha$ is invariant under the flow $\varphi^t$,
so contact flows are always volume preserving.

\subsection{Examples}
\label{s:examples-geodesic}

We now give a few standard examples of Anosov flows.

\subsubsection{Geodesic flows}

Assume that $(\Sigma,g)$ is a compact Riemannian manifold.
We let $M$ be the sphere bundle of $\Sigma$:
$$
M=S\Sigma:=\{(y,w)\in T\Sigma\colon |w|_g=1\}.
$$
The geodesic flow $\varphi^t$ is the flow
on~$M$ defined as follows: if $(y,w)\in S\Sigma$ and
$\gamma:\mathbb R\to \Sigma$ is the geodesic such that
$\gamma(0)=y$, $\dot\gamma(0)=w$, then
$\varphi^t(y,w)=(\gamma(t),\dot\gamma(t))$.
The flow $\varphi^t$ is a contact flow, where
the contact 1-form $\alpha$ on $S\Sigma$ is defined as follows:
$$
\alpha_{(y,w)}(\xi)=\langle d\pi_{(y,w)}\xi, w\rangle_g
$$
where $\pi:S\Sigma\to\Sigma$ is the projection map~-- see for example~\cite[\S1.3.3]{Paternain-Book}.

\begin{prop}
If $(\Sigma,g)$ has everywhere negative sectional curvature, then
the geodesic flow $\varphi^t$ on $M=S\Sigma$ is Anosov.
\end{prop}
For the proof, see for example~\cite[Theorem~3.9.1]{Klingenberg-Book}.

\subsubsection{Suspensions of Anosov maps}

An Anosov map is a discrete time analog of an Anosov flow:
\begin{defi}
Let $\widetilde M$ be a compact manifold and $T:\widetilde M\to\widetilde M$
be a diffeomorphism. We say that $T$ is an \emph{Anosov map}
if the tangent spaces to $\widetilde M$ admit a decomposition
$T_x\widetilde M=E_u(x)\oplus E_s(x)$ which is invariant under $T$,
depends continuously on~$x$, and satisfies the following exponential contraction property
for some constants $C,\theta>0$ and a Riemannian metric on~$\widetilde M$:
$$
|dT^k(x)v|\leq Ce^{-\theta |k|}|v|\quad\text{if}\quad\begin{cases}
v\in E_u(x),& k\leq 0\quad\text{or}\\
v\in E_s(x),& k\geq 0.
\end{cases}
$$
\end{defi}
Basic examples of Anosov maps are the toric automorphisms
$$
T:\mathbb T^d\to\mathbb T^d,\quad
T(x)=Ax\bmod \mathbb Z^d
$$
where $\mathbb T^d=\mathbb R^d/\mathbb Z^d$ is the $d$-dimensional
torus and the matrix $A\in \GL(d,\mathbb Z)$, $|\det A|=1$, has no eigenvalues
on the unit circle.

To make an Anosov map into an Anosov flow, we use suspensions.
Let $T:\widetilde M\to\widetilde M$ be an Anosov map
and $\tau:\widetilde M\to (0,\infty)$ be a smooth function,
called the \emph{roof function} of the suspension.
Let $M$ be the manifold obtained by taking the cylinder
$\{(x,s)\mid x\in \widetilde M,\ 0\leq s\leq \tau(x)\}$
and gluing its two ends by identifying
$(x,\tau(x))$ with $(T(x),0)$. Alternatively, we may define
$M$ as the quotient of $\widetilde M\times\mathbb R$
by the action of $\mathbb Z$ generated by the map
$(x,s+\tau(x))\mapsto (T(x),s)$.
Now, the vector field $X:=\partial_s$ is well-defined on~$M$ and generates an Anosov flow called the \emph{suspension}
of $T$ with roof function $\tau$. Here the Anosov property is easy to check
when $\tau$ is constant and the general case is obtained by a time
change, which does not change the Anosov property~-- see for example~\cite[Proposition~17.4.5]{Katok-Hasselblatt}.

\subsection{Operators and resolvents}

Let $\varphi^t=\exp(tX)$ be an Anosov flow on a manifold~$M$. The vector field
$X$ defines a first order differential operator $X:C^\infty(M)\to C^\infty(M)$.
For $t\in\mathbb R$, define the operator
$$
e^{-tX}:C^\infty(M)\to C^\infty(M),\quad
e^{-tX}f:=f\circ \varphi^{-t}.
$$
The notation $e^{-tX}$ is justified as follows:
for each $f\in C^\infty(M)$ we have
\begin{equation}
\label{e:evolver}
\partial_t (e^{-tX}f)=-e^{-tX} Xf=-Xe^{-tX}f.
\end{equation}
Now, for a complex number $\lambda$ such that
$\Re\lambda>0$ we define the Pollicott--Ruelle resolvent
\begin{equation}
  \label{e:P-R-resolvent}
R_X(\lambda)f:=\int_0^\infty e^{-\lambda t}e^{-tX}f\,dt.
\end{equation}
Here the integral converges exponentially fast in the sup-norm when $f$ is continuous.

We have the identity~\eqref{e:inversor}. Indeed, take $f\in C^\infty(M)$
and assume that $\Re\lambda>0$. Then
$$
R_X(\lambda)(X+\lambda)f=(X+\lambda)R_X(\lambda)f=-\int_0^\infty \partial_t(e^{-\lambda t}e^{-tX}f)\,dt
=f
$$
where in the second equality we consider $X+\lambda$ as a differential operator
on distributions.

\subsubsection{More general operators}
\label{s:more-general-operators}

We now extend the definition of Pollicott--Ruelle resolvent to more general operators.
Let $\mathcal E$ be a (finite dimensional complex) $C^\infty$ vector bundle over $M$.
Denote by $C^\infty(M;\mathcal E)$ the space of smooth sections of $\mathcal E$.
\begin{defi}
\label{d:lift}
An operator $\mathbf X:C^\infty(M;\mathcal E)\to C^\infty(M;\mathcal E)$
is called a \emph{lift} of the vector field $X$ to $\mathcal E$
if
\begin{equation}
\mathbf X(f\mathbf u)=(Xf)\mathbf u+f(\mathbf X\mathbf u)\quad\text{for all}\quad
f\in C^\infty(M;\mathbb C),\quad
u\in C^\infty(M;\mathcal E).
\end{equation}
\end{defi}
If we fix a local frame $\mathbf e_1,\dots,\mathbf e_n\in C^\infty(U;\mathcal E)$
on~$\mathcal E$, where $U\subset M$ is an open set, then lifts of $X$
have the form
\begin{equation}
  \label{e:Xbf-form}
\mathbf X\sum_{j=1}^n f_j(x) \mathbf e_j(x)=\sum_{j=1}^n\bigg(Xf_j(x)+\sum_{k=1}^n A_{jk}(x)f_k(x)
\bigg)\mathbf e_j(x),\quad x\in U
\end{equation}
for all $f_1,\dots,f_n\in C^\infty(M;\mathbb C)$ where $(A_{jk}(x))$ is an $n\times n$ complex matrix with entries which are smooth functions
on~$U$.

We next define parallel transport on $\mathcal E$.
Let $x_0\in M$ and define the curve $x(t):=\varphi^t(x_0)$.
Assume that $\mathbf v(t)\in \mathcal E(x(t))$, $t\in\mathbb R$, is a smooth
section of the pullback of $\mathcal E$ to the curve $x(t)$.
We define the derivative
$D_{\mathbf X} \mathbf v(t)\in \mathcal E(x(t))$
by requiring that
$$
D_{\mathbf X}(\mathbf u(x(t)))=\mathbf X\mathbf u(x(t))\quad\text{for all}\quad
\mathbf u\in C^\infty(M;\mathcal E).
$$
In a local frame we can write
\begin{equation}
  \label{e:covder-cor}
D_{\mathbf X}\sum_{j=1}^n f_j(t)\mathbf e_j(x(t))=\sum_{j=1}^n \bigg(\dot f_j(t)
+\sum_{k=1}^n A_{jk}(x(t))f_k(t)\bigg)\mathbf e_j(x(t)).
\end{equation}
We say that $\mathbf v(t)$ is \emph{parallel} if $D_{\mathbf X}\mathbf v(t)=0$
for all $t$. Using the coordinate expression~\eqref{e:covder-cor}
and the existence/uniqueness theorem for linear systems of ODEs, we see
that for each $\mathbf v_0\in\mathcal E(x(0))$ there exists
a unique parallel field $\mathbf v(t)$ such that $\mathbf v(0)=\mathbf v_0$.
We then define the \emph{parallel transport operator}
\begin{equation}
  \label{e:parallel-def}
\mathscr T_{\mathbf X}^t(x):\mathcal E(x)\to\mathcal E(\varphi^t(x)),\quad
t\in\mathbb R
\end{equation}
such that for any parallel field $\mathbf v(t)$
we have $\mathscr T_{\mathbf X}^t(x(0))\mathbf v(0)=\mathbf v(t)$.

We now define the family of operators
$$
e^{-t\mathbf X}:C^\infty(M;\mathcal E)\to C^\infty(M;\mathcal E),\quad
t\in\mathbb R
$$
so that the evolution equation~\eqref{e:evolver} holds with $X$ replaced by $\mathbf X$.
In terms of parallel transport it can be described as follows:
for each $\mathbf u\in C^\infty(M;\mathcal E)$ and $x\in M$ we have
\begin{equation}
  \label{e:evolution-parallel}
e^{-t\mathbf X}\mathbf u(x)=\mathscr T_{\mathbf X}^t(\varphi^{-t}(x))\mathbf u(\varphi^{-t}(x)).
\end{equation}
We now want to define the Pollicott--Ruelle resolvent of $\mathbf X$ similarly
to~\eqref{e:P-R-resolvent}. For that fix an inner product on the fibers of $\mathcal E$
and take constants $C_{\mathbf X},C_1$ such that
$$
\|\mathscr T_{\mathbf X}^t(x)\|_{\mathcal E(x)\to\mathcal E(\varphi^t(x))}
\leq C_1 e^{C_{\mathbf X}t}\quad\text{for all}\quad
t\geq 0,\quad
x\in M.
$$
Note that the constant $C_1$ depends on the choice of the inner product but
$C_{\mathbf X}$ does not. Now we define
\begin{equation}
  \label{e:P-R-resolvent-gen}
R_{\mathbf X}(\lambda)\mathbf u:=\int_0^\infty e^{-\lambda t}e^{-t\mathbf X}\mathbf u\,dt\quad\text{for}\quad
\Re\lambda>C_{\mathbf X},\quad
\mathbf u\in C^\infty(M;\mathcal E).
\end{equation}
The integral converges in the space of continuous functions $C^0(M;\mathcal E)$.
We have the identities similar to~\eqref{e:inversor}:
\begin{equation}
  \label{e:inversor-gen}
R_{\mathbf X}(\lambda)(\mathbf X+\lambda)\mathbf u
=(\mathbf X+\lambda)R_{\mathbf X}(\lambda)\mathbf u=\mathbf u\quad\text{for all}\quad
\mathbf u\in C^\infty(M;\mathcal E),\quad
\Re\lambda>C_{\mathbf X}.
\end{equation}

\subsubsection{Examples}
\label{s:examples-1}

We now give several natural examples of lifts $\mathbf X$.
First of all, if $\mathcal E=M\times\mathbb C$ is the trivial
line bundle over~$M$, then lifts of $X$ have the form
$$
\mathbf X=X+V\quad\text{for some potential}\quad
V\in C^\infty(M;\mathbb C).
$$
The operator $e^{-t\mathbf X}$ is given by
$$
e^{-t\mathbf X}u(x)=\exp\bigg(-\int_0^t V(\varphi^{-s}(x))\,ds\bigg)u(\varphi^{-t}(x)).
$$
The next example is given by the bundles of differential forms
$$
\Omega^k:=\wedge^k T^*M
$$
and $\mathbf X:=\mathcal L_X$ is the Lie derivative.
In this case the operator $e^{-t\mathbf X}$ is the pullback of differential forms
by $\varphi^{-t}$.

One can also consider the smaller bundle of perpendicular forms
$$
\Omega^k_0:=\{\mathbf u\in \Omega^k\mid \iota_X\mathbf u=0\}
$$
with the same operator $\mathbf X=\mathcal L_X$, which is important for the analysis
of the Ruelle zeta function (see for example~\cite{DZ-zeta}).

We can consider a more general setting by taking a complex vector bundle $\mathcal V$
over $M$ equipped with a flat connection $\nabla$, considering the bundle
$\mathcal E:=\Omega^k\otimes\mathcal V$, and putting
$$
\mathbf X:=\mathcal L_{X,\nabla}=d^\nabla\iota_X+\iota_X d^\nabla
$$
where $d^\nabla:C^\infty(\Omega^k\otimes\mathcal V)\to C^\infty(\Omega^{k+1}\otimes\mathcal V)$
is the twisted exterior derivative associated to $\nabla$.
The resulting Pollicott--Ruelle resonances have important applications
to \emph{Fried's conjecture} relating dynamical zeta functions and torsion~-- see for example Dang--Guillarmou--Rivi\`ere--Shen~\cite[\S3.3]{DGRS}.

A special case of the flat connection example above is when $\mathcal E$
is the orientation bundle of the bundle $E_s$. This bundle can be used to
generalize known results on meromorphic continuation of dynamical zeta functions 
to the case of nonorientable $E_s$~-- see Borns-Weil--Shen~\cite{Borns-Weil-Shen}.

\section{Microlocal framework and the lifted flow}
  \label{s:microlocal-framework}

In this and the next section we assume that $\varphi^t=e^{tX}$
is an Anosov flow on a compact manifold~$M$, $\mathcal E$ is a vector bundle over $M$,
and $\mathbf X:C^\infty(M;\mathcal E)\to C^\infty(M;\mathcal E)$ is a lift of~$X$
in the sense of Definition~\ref{d:lift}. (In particular, this includes the
special scalar case when $\mathcal E=M\times\mathbb C$
and $\mathbf X=X$.)

We henceforth fix a density $\rho_0$ on $M$ and an Hermitian inner product $\langle\bullet,\bullet\rangle_{\mathcal E}$ on the fibers
of~$\mathcal E$, which together fix the inner product on the space
$L^2(M;\mathcal E)$.

We use the semiclassically rescaled version of $\mathbf X$,
$$
\mathbf P:=-ih\mathbf X.
$$
Here $h\in (0,1]$ is a small number called the \emph{semiclassical parameter}. In the
present paper the semiclassical rescaling is a technical tool useful in the proof of the meromorphic
continuation of the Pollicott--Ruelle resolvent, and $h$ will be
ultimately fixed small enough (so that the $\mathcal O(h^\infty)$ remainders
in semiclassical estimates can be removed and Lemma~\ref{l:invertibility} holds). In applications to spectral
gaps (such as the work of Nonnenmacher--Zworski~\cite{Nonnenmacher-Zworski-Inventiones})
one has $h\approx |\Re\lambda|^{-1}$ and studies the limit $h\to 0$.

\subsection{Semiclassical analysis}

We discuss the behavior of $\mathbf P$ from the point of view of microlocal analysis,
more precisely its semiclassical version.
We refer the reader to the book of Zworski~\cite{Zworski-Book} for an introduction to semiclassical analysis
and to the book of Dyatlov--Zworski~\cite[Appendix~E]{DZ-Book} (which builds on~\cite{Zworski-Book}) for some of the more advanced tools used here.

For $m\in\mathbb R$, denote by $S^m_h(T^*M)$ the class of $h$-dependent \emph{Kohn--Nirenberg symbols}
of order $m$ on the cotangent bundle $T^*M$, consisting of $h$-dependent functions
$a(x,\xi;h)\in C^\infty(T^*M)$ satisfying the derivative bounds for all multiindices
$\alpha,\beta$
$$
|\partial^\alpha_x\partial^\beta_\xi a(x,\xi;h)|\leq C_{\alpha\beta}\langle\xi\rangle^{m-|\beta|}\quad\text{for all}\quad (x,\xi)\in T^*M,\quad
0<h\leq 1.
$$
Here $\langle\xi\rangle:=\sqrt{1+|\xi|^2}$. This is the class used in~\cite[\S14.2.2]{Zworski-Book}. The estimates from the book~\cite{DZ-Book}, on which this paper relies,
use instead the smaller class of \emph{polyhomogeneous symbols}
with expansions in powers of $\xi$ and $h$, see~\cite[Definition~E.3]{DZ-Book}.
We will apply these estimates to the conjugated operator $\widetilde{\mathbf P}$
(see~\S\ref{s:conjugated-operator}) which is not polyhomogeneous
and we explain below why the results of~\cite{DZ-Book} still hold.

Denote by $\Psi^m_h(M)$ the class of semiclassical pseudodifferential operators
with symbols in~$S^m_h(T^*M)$. These are $h$-dependent families of operators
on $C^\infty(M)$ and on the space of distributions $\mathcal D'(M)$.
We refer to~\cite[\S14.2.2]{Zworski-Book} and~\cite[\S E.1.7]{DZ-Book} for details.
We use the \emph{semiclassical principal symbol} isomorphism
\begin{equation}
  \label{e:sigma-h-def-1}
\sigma_h:{\Psi^m_h(M)\over h\Psi^{m-1}_h(M)}\ \to\ {S^m_h(T^*M)\over hS^{m-1}_h(T^*M)}. 
\end{equation}
The space $T^*M$ is not compact because $\xi$ is allowed to go to infinity.
We will use the \emph{fiber-radial compactification} $\overline{T^*M}$
obtained by adding to $T^*M$ a sphere at infinity. See for example~\cite[\S E.1.3]{DZ-Book}
for details.

\subsubsection{Operators on sections of vector bundles}

We now discuss the class of semiclassical pseudodifferential operators
$\Psi^m_h(M;\End(\mathcal E))$ acting on the space of sections
$C^\infty(M;\mathcal E)$ of the vector bundle $\mathcal E$.
If $\mathcal E$ is trivial and $n=\dim\mathcal E$, then operators
on $C^\infty(M;\mathcal E)$ are identified with
$n\times n$ matrices of operators on $C^\infty(M)$.
We say such a matrix is in $\Psi^m_h(M;\End(\mathcal E))$
if all of its entries are in $\Psi^m_h(M)$.
This class does not depend on the choice of a (smooth) trivialization
of $\mathcal E$ since composition with multiplication operators
maps $\Psi^m_h(M)$ into itself. Since pseudodifferential operators
are smoothing and rapidly decaying in $h$ away
from the diagonal, one can use the above definition locally
to make sense of $\Psi^m_h(M;\End(\mathcal E))$ for a general
bundle~$\mathcal E$.
See~\cite[Definition~18.1.32]{Hormander3} for more details (in the related
nonsemiclassical setting). Any element of $\Psi^m_h(M;\End(\mathcal E))$
is bounded uniformly in~$h$ in operator norm
$H^s_h(M;\mathcal E)\to H^{s-m}_h(M;\mathcal E)$
where $H^s_h(M;\mathcal E)$ denotes the semiclassical Sobolev
space defined similarly to~\cite[Definition~E.20]{DZ-Book}.

For $\mathbf A\in\Psi^m_h(M;\End(\mathcal E))$, we use the above procedure
and the map~\eqref{e:sigma-h-def-1} to define the semiclassical principal
symbol
$$
\sigma_h(\mathbf A)\in {S^m_h(T^*M;\End(\pi^*\mathcal E))\over hS^{m-1}_h(T^*M;\End(\pi^*\mathcal E))}.
$$
Here $\pi:T^*M\to M$ is the projection map,
$\pi^*\mathcal E$ is the pullback of $\mathcal E$ to a vector bundle
over $T^*M$, and $\End(\pi^*\mathcal E)$ is the bundle of homomorphisms
from $\pi^*\mathcal E$ to itself.
Note that $\sigma_h$ is surjective and $\sigma_h(\mathbf A)=0$ if and only if $\mathbf A\in h\Psi^{m-1}_h(M;\End(\mathcal E))$.

We say $\mathbf A$ is \emph{principally scalar} if $\sigma_h(\mathbf A)$ is scalar,
that is there exists $a\in S^m_h(T^*M)$ such that $\sigma_h(\mathbf A)=aI_{\pi^*\mathcal E}$
modulo $hS^{m-1}_h(T^*M;\End(\pi^*\mathcal E))$. In this case we treat
$\sigma_h(\mathbf A)$ as a scalar function on $T^*M$ by identifying it with
(the equivalence class of) $a$.

Using the standard algebraic properties of the scalar calculus $\Psi^m_h(M)$
(see for instance~\cite[Theorem~14.1]{Zworski-Book} and~\cite[Proposition~E.17]{DZ-Book})
we obtain the following properties of the calculus $\Psi^m_h(M;\End(\mathcal E))$:

\begin{itemize}
\item Product Rule: if $\mathbf A\in\Psi^m_h(M;\End(\mathcal E))$ and $\mathbf B\in\Psi^\ell_h(M;\End(\mathcal E))$,
then
\begin{equation}
  \label{e:product-rule}
\mathbf{AB}\in\Psi^{m+\ell}_h(M;\End(\mathcal E)),\qquad
\sigma_h(\mathbf {AB})=\sigma_h(\mathbf A)\sigma_h(\mathbf B)
\end{equation}
where the right-hand side is understood as composition of sections of $\End(\pi^*\mathcal E)$.
\item Commutator Rule: if $\mathbf A\in\Psi^m_h(M;\End(\mathcal E))$, $\mathbf B\in\Psi^\ell_h(M;\End(\mathcal E))$ are both \emph{principally scalar}, then,
with $\{\bullet,\bullet\}$ denoting the Poisson bracket on~$T^*M$,
\begin{equation}
  \label{e:commutator-general}
[\mathbf A,\mathbf B]\in h\Psi^{m+\ell-1}_h(M;\End(\mathcal E)),\qquad
\sigma_h(h^{-1}[\mathbf A,\mathbf B])=-i\{\sigma_h(\mathbf A),\sigma_h(\mathbf B)\}.
\end{equation}
\item Adjoint Rule: if $\mathbf A\in\Psi^m_h(M;\End(\mathcal E))$, then
its formal adjoint $\mathbf A^*$ satisfies
$$
{\mathbf A}^*\in \Psi^m_h(M;\End(\mathcal E)),\qquad
\sigma_h(\mathbf A^*)=\sigma_h(\mathbf A)^*
$$
where the right-hand side is defined using the adjoint operation
on $\End(\pi^*\mathcal E)$ induced by the inner product
$\langle\bullet,\bullet\rangle_{\mathcal E}$.
\end{itemize}
We next discuss the \emph{wavefront set} and the \emph{elliptic set}
of an operator $\mathbf A\in\Psi^m_h(M;\mathcal E)$. The wavefront set
$\WFh(\mathbf A)$ is a compact subset of $\overline{T^*M}$ giving the essential
support of the full symbol of~$\mathbf A$. In terms of the wavefront set
of scalar pseudodifferential operators (see for example~\cite[Definition~E.27]{DZ-Book}), we define $\WFh(\mathbf A)$
as the union of the wavefront sets of the entries of $\mathbf A$ as an $n\times n$
matrix of operators, with respect to any trivialization of~$\mathcal E$.

The elliptic set $\Ell_h(\mathbf A)$ is the open subset of $\overline{T^*M}$ on which the principal symbol $\sigma_h(\mathbf A)$ is essentially invertible
(as an endomorphism of~$\pi^*\mathcal E$).
More precisely, a point $(x_0,\xi_0)\in \overline{T^*M}$ lies in $\Ell_h(\mathbf A)$
if there exists a constant~$C$ such that we have $\big\|\big(\sigma_h(\mathbf A)(x,\xi)\big)^{-1}\big\|\leq C\langle\xi\rangle^{-m}$
for all sufficiently small~$h$ and all $(x,\xi)$ in some neighborhood of $(x_0,\xi_0)$
in $\overline{T^*M}$.

Finally, we give the following version of sharp G\r arding inequality
for pseudodifferential operators on vector bundles.
It is an analog of~\cite[Proposition~E.34]{DZ-Book}
but we restrict a simpler case, putting $B:=0$ and considering
a special subclass of nonnegative symbols in $C^\infty(T^*M;\End(\pi^*\mathcal E))$.
\begin{lemm}
  \label{l:sharp-Garding}
Assume that $\mathbf A\in\Psi^{2m}_h(M;\End(\mathcal E))$ and
$\mathbf B_1\in\Psi^0_h(M;\End(\mathcal E))$ satisfy $\WFh(\mathbf A)\subset\Ell_h(\mathbf B_1)$.
Assume moreover that the principal symbol $\sigma_h(\mathbf A)$
has the form
\begin{equation}
  \label{e:symbol-assume-more}
\sigma_h(\mathbf A)=\chi\mathbf a,\quad
\chi\in C^\infty(\overline{T^*M}),\quad
\chi\geq 0,\quad
\mathbf a\in S^{2m}(T^*M;\End(\pi^*\mathcal E))
\end{equation}
where $\chi,\mathbf a$ are $h$-independent and
$\Re\mathbf a$ is uniformly positive definite
on some neighborhood $V\subset \overline{T^*M}$ of~$\supp\chi$, that is there exists a constant $c>0$ such that
$$
\Re\langle\mathbf a(x,\xi)\mathbf v,\mathbf v\rangle_{\mathcal E(x)}\geq c\langle\xi\rangle^{2m}\|\mathbf v\|_{\mathcal E(x)}^2\quad\text{for all}\quad
(x,\xi)\in V,\quad
\mathbf v\in\mathcal E(x).
$$
Then there exists a constant $C$ such that for each $N$, all small~$h$, and
all $\mathbf u\in H^m(M;\mathcal E)$
\begin{equation}
  \label{e:sharp-Garding}
\Re\langle \mathbf A\mathbf u,\mathbf u\rangle_{L^2(M;\mathcal E)}
\geq -Ch\|\mathbf B_1\mathbf u\|_{H^{m-{1\over 2}}_h}^2-\mathcal O(h^\infty)\|\mathbf u\|_{H^{-N}_h}^2.
\end{equation}
\end{lemm}
\Remark In fact~\eqref{e:symbol-assume-more} can be replaced by the weaker and more
natural assumption that $\Re\sigma_h(\mathbf A)$ is nonnegative everywhere,
see~\cite[Remark~2 on p.79]{Hormander3} for the nonsemiclassical case.
Instead of establishing a semiclassical version of this result here, we choose to
make the stronger assumption~\eqref{e:symbol-assume-more} which allows us
to use the scalar sharp G\r arding inequality as a black box.
\begin{proof}
If $\sigma_h(\mathbf A)=0$, then $\mathbf A\in h\Psi^{2m-1}_h(M;\End(\mathcal E))$
so~\eqref{e:sharp-Garding} holds
by the elliptic estimate (see~\S\ref{s:elliptic} below) since $\WFh(\mathbf A)\subset\Ell_h(\mathbf B_1)$. Therefore, we may replace $\mathbf A$ with any other operator with the same principal symbol and wavefront set contained in~$\Ell_h(\mathbf B_1)$.
Moreover, from the Adjoint Rule above we see
that one may replace $\mathbf a$ by $\Re\mathbf a:={1\over 2}(\mathbf a+\mathbf a^*)$.
We thus henceforth assume that $\mathbf a$ is self-adjoint.
Since $\WFh(\mathbf A)\subset\Ell_h(\mathbf B_1)$, we may also assume that
$\supp\chi\subset\Ell_h(\mathbf B_1)$.

Since $\mathbf a$ is positive definite on $V\Supset\supp\chi$, we may write
$$
\mathbf a=\mathbf f^*\mathbf f\quad\text{near }\supp\chi\quad\text{for some}\quad
\mathbf f\in S^m(T^*M;\End(\pi^*\mathcal E)),\quad
\supp\mathbf f\subset\Ell_h(\mathbf B_1).
$$
For example, we may take $\chi'\in C^\infty(\overline{T^*M})$ such that
$\chi'=1$ near $\supp \chi$ and $\supp\chi'\subset V\cap\Ell_h(\mathbf B_1)$, and put
$\mathbf f:=\chi'\sqrt{\mathbf a}$.

Using a partition of unity on~$\chi$, we reduce to a case when $\chi$ is supported
in some open set over which $\mathcal E$ is trivialized by some orthonormal frame.
Using that frame, we may consider the pseudodifferential
operator $\Op_h(\chi)\in \Psi^0_h(M)$ as
an operator on sections of $\mathcal E$. Take
$\mathbf F\in \Psi^m_h(T^*M;\End(\mathcal E))$ with principal symbol
$\mathbf f$ and $\WFh(\mathbf F)\subset\Ell_h(\mathbf B_1)$, then $\sigma_h(\mathbf A)=\sigma_h(\mathbf F^*\Op_h(\chi)\mathbf F)$,
so we may assume that $\mathbf A=\mathbf F^*\Op_h(\chi)\mathbf F$.
Now
$$
\begin{aligned}
\langle \mathbf A\mathbf u,\mathbf u\rangle_{L^2(M;\mathcal E)}
&=\langle\Op_h(\chi)\mathbf F\mathbf u,\mathbf F\mathbf u\rangle_{L^2(M;\mathcal E)}
\geq -Ch\|\mathbf F\mathbf u\|^2_{H^{-{1\over 2}}_h}\\
&\geq -Ch\|\mathbf B_1\mathbf u\|^2_{H^{m-{1\over 2}}_h}
-\mathcal O(h^\infty)\|\mathbf u\|_{H^{-N}_h}
\end{aligned}
$$
Here in the first inequality we use that $\chi\geq 0$ and 
apply the scalar sharp G\r arding inequality~\cite[Proposition~E.23]{DZ-Book} for the operator $\Op_h(\chi)$.
In the last inequality we use the elliptic estimate.
\end{proof}

\subsection{Semiclassical properties of $\mathbf P$}

The operator $\mathbf P=-ih\mathbf X$ is a semiclassical differential operator in the class $\Psi^1_h(M;\End(\mathcal E))$, as follows from~\eqref{e:Xbf-form}. It is principally
scalar with the principal symbol given by
$$
p(x,\xi):=\langle \xi,X(x)\rangle,\quad
x\in M,\quad
\xi\in T_x^*M.
$$
Note that $p$ is real valued and homogeneous of degree~1 in~$\xi$.

\subsubsection{The lifted flow}
\label{s:dynamics}

For semiclassical estimates, it is important to understand the characteristic surface
$\{p=0\}\subset \overline{T^*M}$ and the Hamiltonian flow $e^{tH_p}$ on this surface.
For that we introduce the \emph{dual flow/unstable/stable decomposition}
of the fibers of the cotangent bundle $T^*M$:
\begin{equation}
  \label{e:stun-dual}
T^*_x M=E_0^*(x)\oplus E_u^*(x)\oplus E_s^*(x),\quad
x\in M
\end{equation}
which is defined in terms of the original flow/unstable/stable decomposition~\eqref{e:stun-original}
as follows:
$$
E_0^*:=(E_u\oplus E_s)^\perp,\qquad
E_u^*:=(E_0\oplus E_u)^\perp,\qquad
E_s^*:=(E_0\oplus E_s)^\perp.
$$
Any continuous subbundle of $T^*M$ can be considered as a closed subset of $\overline{T^*M}$,
and the characteristic surface of~$p$ is
$$
\{p=0\}=\{(x,\xi)\in \overline {T^*M}\mid \langle\xi,X(x)\rangle=0\}=E_u^*\oplus E_s^*.
$$
Next, the Hamiltonian flow of $p$ has the form
$$
e^{tH_p}(x,\xi)=(\varphi^t(x),d\varphi^t(x)^{-T}\xi)
$$
and extends to a smooth flow on $\overline{T^*M}$.
Here $d\varphi^t(x)^{-T}:T_x^* M\to T_{\varphi^t(x)}^*M$
is the inverse of the transpose of $d\varphi^t(x):T_xM\to T_{\varphi^t(x)}M$.

Following~\cite[(E.1.11)]{DZ-Book}, denote by
$$
\kappa:T^*M\setminus 0\to \partial\overline{T^*M}
$$
the canonical projection to fiber infinity $\partial\overline{T^*M}$.
Then $\kappa(E_u^*),\kappa(E_s^*)$ are compact subsets of $\partial\overline{T^*M}$
invariant under the flow $e^{tH_p}$.

The Anosov property~\eqref{e:Anosov-contraction} carries over to the decomposition~\eqref{e:stun-dual} as follows:
if $|\bullet|$ denotes some smooth norm on the fibers of $T^*M$, then
$$
|e^{tH_p}(x,\xi)|\leq Ce^{-\theta|t|}|\xi|\quad\text{if}\quad
\begin{cases}
\xi\in E_u^*(x),& t\leq 0\quad\text{or}\\
\xi\in E_s^*(x),& t\geq 0.
\end{cases}
$$
Moreover, if $\xi\in E_0^*(x)$, then $|e^{tH_p}(x,\xi)|\leq C|\xi|$ for all $t$.
This implies the following global dynamical properties
of the flow~$e^{tH_p}$ on $\overline{T^*M}$:
\begin{itemize}
\item if $(x,\xi)\in\overline{T^*M}\setminus(E_0^*\oplus E_s^*)$, then as $t\to\infty$,
$e^{tH_p}(x,\xi)$ converges to $\kappa(E_u^*)$ (in the topology of $\overline{T^*M}$) and $|e^{tH_p}(x,\xi)|\to\infty$ exponentially fast;
\item if $(x,\xi)\in\overline{T^*M}\setminus(E_0^*\oplus E_u^*)$, then
as $t\to-\infty$, $e^{tH_p}(x,\xi)$ converges to $\kappa(E_s^*)$
and $|e^{tH_p}(x,\xi)|\to\infty$ exponentially fast.
\end{itemize}
Indeed, to show for example the first statement we may write
$\xi=\xi_0+\xi_u+\xi_s$ where $\xi_0\in E_0^*(x)$, $\xi_u\in E_u^*(x)$,
$\xi_s\in E_s^*(x)$ and $\xi_u\neq 0$. Then as $t\to\infty$,
$e^{tH_p}(x,\xi_0+\xi_s)$ stays bounded while
$e^{tH_p}(x,\xi_u)$ grows exponentially and thus is the dominant
component of~$e^{tH_p}(x,\xi)$.

The above statements are locally uniform in $(x,\xi)$. They imply in particular
that $\kappa(E_u^*)$ is a \emph{radial sink} and $\kappa(E_s^*)$ is a \emph{radial source}
for the flow $e^{tH_p}$ in the sense of~\cite[Definition~E.50]{DZ-Book}.
They also give the following statement about the flow on the characteristic set:
\begin{lemm}
\label{l:flow-char-global}
Fix arbitrary neighborhoods $V_u,V_s,V_0$ of $\kappa(E_u^*)$, $\kappa(E_s^*)$,
and the zero section in $\overline{T^*M}$. Let $(x,\xi)\in \{p=0\}\subset\overline{T^*M}$. Then:
\begin{itemize}
\item if $(x,\xi)\not\in \kappa(E_s^*)$,
then there exists $t\geq 0$ such that $e^{tH_p}(x,\xi)\in V_u\cup V_0$;
\item if $(x,\xi)\not\in \kappa(E_u^*)$, then
there exists $t\leq 0$ such that $e^{tH_p}(x,\xi)\in V_s\cup V_0$.
\end{itemize}
\end{lemm}
\begin{proof}
We only show the first statement.
If $(x,\xi)\not\in E_s^*$ then, since $(x,\xi)\in \{p=0\}$, we have
$(x,\xi)\not\in E_0^*\oplus E_s^*$, so $e^{tH_p}(x,\xi)$ converges to $\kappa(E_u^*)$
as $t\to\infty$. Thus $e^{tH_p}(x,\xi)\in V_u$ for $t\geq 0$ large enough.
Now, if $(x,\xi)\in E_s^*$ and $\xi$ is finite, then
$e^{tH_p}(x,\xi)$ converges to the zero section as $t\to\infty$.
Thus $e^{tH_p}(x,\xi)\in V_0$ for $t\geq 0$ large enough.
\end{proof}

\subsubsection{Weight functions}

The dynamical properties of the flow $e^{tH_p}$ discussed in~\S\ref{s:dynamics}
make it possible to construct weight functions decaying along this flow,
which are used later to define the anisotropic Sobolev spaces:
\begin{lemm}
\label{l:weight}
Fix some real numbers $m_u\leq m_0\leq m_s$ and conic neighborhoods
$V_u,V_s\subset T^*M\setminus 0$ of $E_u^*,E_s^*$.
Then there exists a function $\mathfrak m\in C^\infty(T^*M\setminus 0;\mathbb R)$ such that:
\begin{itemize}
\item $\mathfrak m(x,\xi)$ is positively homogeneous of degree~0 in~$\xi$;
\item $m_u\leq \mathfrak m\leq m_s$ everywhere;
\item $\mathfrak m=m_u$ in some conic neighborhood of $E_u^*$;
\item $\mathfrak m=m_s$ in some conic neighborhood of $E_s^*$;
\item $\mathfrak m=m_0$ outside of $V_u\cup V_s$;
\item $H_p \mathfrak m\leq 0$ everywhere.
\end{itemize}
\end{lemm}
\Remark A more refined version of Lemma~\ref{l:weight}
can be found in~\cite[Lemma~1.2]{Faure-Sjostrand}.
In the present paper we do not use that $\mathfrak m=m_0$
outside of $V_u\cup V_s$, but it is a convenient
property to have for wavefront set analysis, see~\cite[Theorem~1.7]{Faure-Sjostrand}.
\begin{proof}
A positively homogeneous function of degree~0 on $T^*M\setminus 0$ is
the pullback by~$\kappa$ of a function on the fiber infinity $\partial\overline{T^*M}$,
and $V_u,V_s$ are preimages by $\kappa$ of some neighborhoods
of $\kappa(E_u^*),\kappa(E_s^*)$. Moreover, the flow $e^{tH_p}$ commutes with $\kappa$.
Thus we will construct $\mathfrak m$ as a function on $\partial\overline{T^*M}$,
consider $V_u,V_s$ as open subsets of $\partial\overline{T^*M}$, and
work with the flow $e^{tH_p}$ restricted to $\partial\overline{T^*M}$.

We now construct dynamically adapted cutoffs on $V_u,V_s$ following a standard argument presented for example in~\cite[Lemma~E.53]{DZ-Book}. We shrink $V_u$ if necessary so that
it does not intersect $\kappa(E_0^*\oplus E_s^*)$.
Take $\psi_u\in \CIc(V_u;[0,1])$ such that $\psi_u=1$ near $\kappa(E_u^*)$.
Since $\kappa(E_u^*)$ is a radial sink for the flow~$e^{tH_p}$, there exists
$T>0$ such that
\begin{equation}
  \label{e:contractor}
e^{tH_p}(\supp\psi_u)\subset \{\psi_u=1\}\quad\text{for all}\quad t\geq T.
\end{equation}
Put
$$
\chi_u:={1\over T}\int_T^{2T} \psi_u\circ e^{-tH_p}\,dt\ \in\ C^\infty(\partial\overline{T^*M};[0,1]).
$$
Then $\supp\chi_u\subset V_u$ (as follows from~\eqref{e:contractor}) and $\chi_u=1$ near $\kappa(E_u^*)$. Moreover
$$
H_p\chi_u=-{1\over T}\int_T^{2T} \partial_t(\psi_u\circ e^{-tH_p})\,dt=
{\psi_u\circ e^{-TH_p}-\psi_u\circ e^{-2TH_p}\over T}\geq 0
$$
where we again use~\eqref{e:contractor}:
for each $(x,\xi)\in \partial\overline{T^*M}$
we have $\psi_u(e^{-2TH_p}(x,\xi))=0$ or $\psi_u(e^{-TH_p}(x,\xi))=1$.

A similar argument gives a function
$\chi_s\in\CIc(V_s;[0,1])$ such that
$\chi_s=1$ near $\kappa(E_s^*)$ and $H_p\chi_s\leq 0$ everywhere. It remains to put
$$
\mathfrak m:=(m_u-m_0)\chi_u+(m_s-m_0)\chi_s+m_0.\qedhere
$$
\end{proof}

\subsubsection{Computing the adjoint-commutator}
\label{s:adj-comm}

We now give the following lemma which computes the expression
in the positive commutator argument for the radial estimates in~\S\ref{s:radial} below:
\begin{lemm}
  \label{l:commutator-computed}
Assume that $\mathbf W\in \Psi^{2m}_h(M;\End(\mathcal E))$
and $\mathbf W^*=\mathbf W$. Then there exists $\mathbf Z\in\Psi^{2m}_h(M;\End(\mathcal E))$
such that $\mathbf Z^*=\mathbf Z$, $\WFh(\mathbf Z)\subset\WFh(\mathbf W)$, and
for each $\lambda\in \mathbb C$ and $\mathbf u\in C^\infty(M;\End(\mathcal E))$
\begin{equation}
  \label{e:commcomp-1}
\Im\langle (\mathbf P-ih\lambda)\mathbf u,\mathbf W\mathbf u\rangle_{L^2(M;\mathcal E)}
=h\big\langle(\mathbf Z-(\Re\lambda)\mathbf W) \mathbf u,\mathbf u\big\rangle_{L^2(M;\mathcal E)}.
\end{equation}
Moreover, the semiclassical principal symbol of $\mathbf Z$ is given by
\begin{equation}
  \label{e:commcomp-2}
\sigma_h(\mathbf Z)={1\over 2}H_{\mathbf X}\sigma_h(\mathbf W)
\end{equation}
where $H_{\mathbf X}:C^\infty(T^*M;\End(\pi^*\mathcal E))\to C^\infty(T^*M;\End(\pi^*\mathcal E))$ is a lift of the vector field~$H_p$ (see Definition~\ref{d:lift}
which can be applied to any vector field).

Finally, the evolution group $e^{tH_{\mathbf X}}$ is
described in terms of the parallel transport from~\eqref{e:parallel-def}:
\begin{equation}
  \label{e:commcomp-3}
e^{tH_{\mathbf X}}\mathbf w(x,\xi)=|\det d\varphi^t(x)|(\mathscr T_{\mathbf X}^t(x))^*\mathbf w(e^{tH_p}(x,\xi))\mathscr T_{\mathbf X}^t(x)
\end{equation}
for all $\mathbf w\in C^\infty(T^*M;\End(\pi^*\mathcal E))$.
Here the adjoint is taken with respect to the inner product~$\langle\bullet,\bullet\rangle_{\mathcal E}$ on~$\mathcal E$ and the determinant
is taken with respect to the density~$\rho_0$ fixed in the beginning of~\S\ref{s:microlocal-framework}.
\end{lemm}
\begin{proof}
1. A direct computation shows that~\eqref{e:commcomp-1} holds with
$$
\mathbf Z:={i\over 2h}(\mathbf P^*\mathbf W-\mathbf W\mathbf P).
$$
Since $\mathbf P\in\Psi^1_h(M;\End(\mathcal E))$ is principally
scalar with real-valued principal symbol, the principal
symbol of $\mathbf P^*\mathbf W-\mathbf W\mathbf P$ is equal to~0.
Thus $\mathbf Z\in\Psi^{2m}_h(M;\End(\mathcal E))$.
From the definition of~$\mathbf Z$ we see also that
$\mathbf Z^*=\mathbf Z$ and
$\WFh(\mathbf Z)\subset \WFh(\mathbf W)$.

\noindent 2. Fix a frame $\mathbf e_1,\dots,\mathbf e_n$
on $\mathcal E$ over some open set $U\subset M$
which is orthonormal with respect to the inner product $\langle\bullet,\bullet\rangle_{\mathcal E}$.
The operator $\mathbf X$ is given by~\eqref{e:Xbf-form}
for some matrix $A(x)=(A_{jk}(x))_{j,k=1}^n$ depending
on $x\in U$, so the operator $\mathbf P$ is given by
$$
\mathbf P\sum_{j=1}^n f_j\mathbf e_j=-ih\sum_{j=1}^n\bigg(
Xf_j+\sum_{k=1}^n A_{jk}f_k\bigg)\mathbf e_j.
$$
Denoting by $\Div_{\rho_0}X:=\rho_0^{-1}\mathcal L_X\rho_0$ the divergence of the vector
field~$X$ with respect to the density~$\rho_0$, we compute the adjoint operator:
$$
\mathbf P^*\sum_{j=1}^n f_j\mathbf e_j=-ih\sum_{j=1}^n\bigg(
(X+\Div_{\rho_0}X)f_j-\sum_{k=1}^n \overline{A_{kj}}f_k\bigg)\mathbf e_j.
$$
Using this we see that~\eqref{e:commcomp-2} holds with
\begin{equation}
  \label{e:commcomp-int-1}
H_{\mathbf X}\mathbf w(x,\xi)=H_p\mathbf w(x,\xi)+\Div_{\rho_0}X(x)\mathbf w(x,\xi)
-A(x)^*\mathbf w(x,\xi)-\mathbf w(x,\xi)A(x)
\end{equation}
where we identify sections of $\End(\pi^*\mathcal E)$
with $n\times n$ matrices using the frame $\mathbf e_1,\dots,\mathbf e_n$
and $H_p$ on the right-hand side acts on each matrix entry separately.

\noindent 3. The operator defined by~\eqref{e:commcomp-3} forms a group in~$t$,
so it suffices to check that for each $\mathbf w\in C^\infty(T^*M;\End(\pi^*\mathcal E))$
we have
\begin{equation}
  \label{e:commcomp-int-2}
\partial_t|_{t=0}\big(|\det d\varphi^t(x)|(\mathscr T_{\mathbf X}^t(x))^*\mathbf w(e^{tH_p}(x,\xi))\mathscr T_{\mathbf X}^t(x)\big)=H_{\mathbf X}\mathbf w(x,\xi).
\end{equation}
We argue in a local frame as in Step~2 above. Using this frame 
we view $\mathscr T_{\mathbf X}^t(x)$ as an $n\times n$ matrix. Using the definition
of parallel transport (see~\eqref{e:parallel-def})
and the formula~\eqref{e:covder-cor} we see that
$$
\partial_t|_{t=0} \mathscr T_{\mathbf X}^t(x)=-A(x).
$$
We also have $\partial_t|_{t=0}\det d\varphi^t(x)=\Div_{\rho_0}X(x)$.
Using these two identities and~\eqref{e:commcomp-int-1}, we verify that~\eqref{e:commcomp-int-2} holds.
\end{proof}

\subsection{The threshold conditions and existence of multipliers}
\label{s:threshold}

We now introduce the threshold regularity conditions needed for the
proof of the Fredholm property of $\mathbf P-ih\lambda$;
more specifically, they are used in the proofs of the radial estimates
in~\S\ref{s:radial} below. We start with the following

\begin{defi}
  \label{d:growth-factor}
Assume that $m_u\leq 0\leq m_s$ are given constants. Define the
\emph{growth factors} $r_u(m_u),r_s(m_s)\in\mathbb R$ as the smallest numbers
such that for each $\varepsilon>0$ there exists a constant $C_\varepsilon>0$ such that
for all $x\in M$ and $t\geq 0$
\begin{equation}
  \label{e:growth-factor}
\begin{aligned}
|\det d\varphi^t(x)|^{1\over 2}\cdot\|\mathscr T_{\mathbf X}^t(x)\|
\cdot \big\|d\varphi^t(x)^{T}|_{E_u^*}\big\|^{-m_u}&\leq C_\varepsilon e^{(r_u(m_u)+\varepsilon)t},\\
|\det d\varphi^t(x)|^{1\over 2}\cdot\|\mathscr T_{\mathbf X}^t(x)\|
\cdot \big\|d\varphi^t(x)^{-T}|_{E_s^*}\big\|^{m_s}&\leq C_\varepsilon e^{(r_s(m_s)+\varepsilon)t}.
\end{aligned}
\end{equation}
\end{defi}
\Remarks 1. The constants $r_u(m_u),r_s(m_s)$ do not depend
on the choice of the inner product on~$\mathcal E$, the metric on~$M$,
and the density $\rho_0$ used to define the norms in~\eqref{e:growth-factor}.

\noindent 2. The bounds~\eqref{e:growth-factor} can be reformulated
in terms of the action of $d\varphi^t$ on the spaces~$E_u,E_s$:
for all $x\in M$ and $t\geq 0$
\begin{equation}
  \label{e:growth-factor-2}
\begin{aligned}
|\det d\varphi^t(x)|^{1\over 2}\cdot\|\mathscr T_{\mathbf X}^t(x)\|
\cdot \big\|d\varphi^t(x)|_{E_s}\big\|^{-m_u}&\leq C_\varepsilon e^{(r_u(m_u)+\varepsilon)t},\\
|\det d\varphi^t(x)|^{1\over 2}\cdot\|\mathscr T_{\mathbf X}^t(x)\|
\cdot \big\|d\varphi^t(x)^{-1}|_{E_u}\big\|^{m_s}&\leq C_\varepsilon e^{(r_s(m_s)+\varepsilon)t}.
\end{aligned}
\end{equation}
We see that
\begin{equation}
  \label{e:rus-big}
r_u(m_u)\leq C_1+\theta m_u,\quad
r_s(m_s)\leq C_1-\theta m_s
\end{equation}
for some constant $C_1$ depending only on the lift~$\mathbf X$, where $\theta>0$ is the constant in the exponential contraction property~\eqref{e:Anosov-contraction}.

The next lemma introduces the threshold regularity conditions
and constructs the multipliers used in the proofs of the radial estimates:
\begin{lemm}
  \label{l:multipliers}
Assume that $m_u\leq 0\leq m_s$ and $\lambda\in\mathbb C$ satisfy the threshold condition
\begin{equation}
  \label{e:threshold-mult}
r_u(m_u)<\Re\lambda,\qquad
r_s(m_s)<\Re\lambda.
\end{equation}
Then there exist $\mathbf w_u,\mathbf w_s\in C^\infty(T^*M\setminus 0;\End(\pi^*\mathcal E))$
such that:
\begin{itemize}
\item $\mathbf w_u^*=\mathbf w_u$, $\mathbf w_s^*=\mathbf w_s$,
and $\mathbf w_u,\mathbf w_s$ are positive definite everywhere;
\item $\mathbf w_u$, $\mathbf w_s$ are positively homogeneous
of degrees $2m_u,2m_s$, that is for each $(x,\xi)\in T^*M\setminus 0$
and $\tau>0$
$$
\mathbf w_u(x,\tau\xi)=\tau^{2m_u}\mathbf w_u(x,\xi),\qquad
\mathbf w_s(x,\tau\xi)=\tau^{2m_s}\mathbf w_s(x,\xi);
$$
\item if $H_{\mathbf X}$ is the operator defined in Lemma~\ref{l:commutator-computed},
then
$$
(H_{\mathbf X}-2\Re\lambda)\mathbf w_u(x,\xi),\qquad
(H_{\mathbf X}-2\Re\lambda)\mathbf w_s(x,\xi)
$$
are self-adjoint, positively homogeneous of degrees $2m_u$, $2m_s$ respectively,
and negative definite for all $(x,\xi)$ in $E_u^*\setminus 0$ and $E_s^*\setminus 0$ respectively.
\end{itemize}
\end{lemm}
\begin{proof}
For two self-adjoint elements $\mathbf a,\mathbf b\in \End(\mathcal E(x))$, we write
$\mathbf a\boldsymbol<\mathbf b$ if $\mathbf b-\mathbf a$ is positive definite.

\noindent 1. Fix a metric on $M$ and define the sections
$\mathbf w_u^{0},\mathbf w_s^{0}\in C^\infty(T^*M\setminus 0;\End(\pi^*\mathcal E))$ by
$$
\mathbf w_u^{0}(x,\xi):=|\xi|^{2m_u} I_{\mathcal E(x)},\quad
\mathbf w_s^{0}(x,\xi):=|\xi|^{2m_s} I_{\mathcal E(x)}
$$
where $I_{\mathcal E(x)}$ is the identity map in $\End(\mathcal E(x))$.
We claim that under the threshold condition~\eqref{e:threshold-mult} we have for all $t>0$ large enough
\begin{equation}
  \label{e:multipliers-int}
\begin{aligned}
e^{tH_{\mathbf X}}\mathbf w_u^0(x,\xi)&\boldsymbol <e^{2\Re\lambda t} \mathbf w_u^0(x,\xi)\quad\text{for all}\quad(x,\xi)\in E_u^*\setminus 0,\\
e^{tH_{\mathbf X}}\mathbf w_s^0(x,\xi)&\boldsymbol<e^{2\Re\lambda t} \mathbf w_s^0(x,\xi)\quad\text{for all}\quad (x,\xi)\in E_s^*\setminus 0.
\end{aligned}
\end{equation}
We show the first statement in~\eqref{e:multipliers-int}, with the second one proved
similarly. Let $(x,\xi)\in E_u^*\setminus 0$ and
$\mathbf v\in \mathcal E(x)$. Using the formula~\eqref{e:commcomp-3}
for $e^{tH_{\mathbf X}}$ we compute
$$
\begin{aligned}
\langle e^{tH_{\mathbf X}}\mathbf w_u^0(x,\xi) \mathbf v,\mathbf v\rangle_{\mathcal E}&=
|\det d\varphi^{t}(x)|\cdot |d\varphi^{t}(x)^{-T}\xi|^{2m_u}\cdot \|\mathscr T_{\mathbf X}^{t}(x)\mathbf v\|_{\mathcal E(\varphi^t(x))}^2,\\
\langle \mathbf w_u^0(x,\xi)\mathbf v,\mathbf v\rangle_{\mathcal E}&=
|\xi|^{2m_u}\|\mathbf v\|_{\mathcal E(x)}^2.
\end{aligned}
$$
We have $m_u\leq 0$ and $|d\varphi^{t}(x)^{-T}\xi|\geq \|d\varphi^{t}(x)^T|_{E_u^*}\|^{-1}\cdot |\xi|$, so
$$
|d\varphi^{t}(x)^{-T}\xi|^{2m_u}\leq \|d\varphi^{t}(x)^T|_{E_u^*}\|^{-2m_u}\cdot
|\xi|^{2m_u}.
$$
Now~\eqref{e:multipliers-int} follows from the bound
$$
|\det d\varphi^{t}(x)|\cdot \|d\varphi^{t}(x)^T|_{E_u^*}\|^{-2m_u}
\|\mathscr T_{\mathbf X}^{t}(x)\|^2< e^{2\Re\lambda t}
$$
which holds for $t>0$ large enough by~\eqref{e:threshold-mult}
since the left-hand side is $\mathcal O_\varepsilon(e^{(2r_u(m_u)+\varepsilon)t})$
for any $\varepsilon>0$.

\noindent 2. Fix $t_0>0$ such that~\eqref{e:multipliers-int} holds with $t:=t_0$. We define
$$
\mathbf w_u:=\int_0^{t_0} e^{-2\Re\lambda t}e^{tH_{\mathbf X}}\mathbf w^0_u\,dt,\qquad
\mathbf w_s:=\int_0^{t_0} e^{-2\Re\lambda t}e^{tH_{\mathbf X}}\mathbf w^0_s\,dt.
$$
It is straightforward to check using~\eqref{e:commcomp-3} that
$\mathbf w_u,\mathbf w_s$ are self-adjoint, positively homogeneous
of degrees $2m_u,2m_s$ respectively, and positive definite.
Since $e^{tH_{\mathbf X}}$ is the evolution group associated to~$H_{\mathbf X}$, we have
$$
(H_{\mathbf X}-2\Re\lambda)\mathbf w_u=\int_0^{t_0} \partial_t\big(e^{-2\Re\lambda t}e^{tH_{\mathbf X}}\mathbf w^0_u\big)\,dt=e^{-2\Re\lambda t_0}e^{t_0H_{\mathbf X}}\mathbf w_u^0
-\mathbf w_u^0
$$
and similarly $(H_{\mathbf X}-2\Re\lambda)\mathbf w_s=e^{-2\Re\lambda t_0}e^{t_0H_{\mathbf X}}\mathbf w_s^0 -\mathbf w_s^0$.
We see that $(H_{\mathbf X}-2\Re\lambda)\mathbf w_u$, $(H_{\mathbf X}-2\Re\lambda)\mathbf w_s$
are self-adjoint and positively homogeneous of degrees~$2m_u$,
$2m_s$ respectively. Moreover, by~\eqref{e:multipliers-int}
these sections are negative definite on $E_u^*\setminus 0,E_s^*\setminus 0$ respectively.
\end{proof}

\subsubsection{Examples}
\label{s:examples-2}

We now compute the growth factors $r_u(m_u)$, $r_s(m_s)$ from Definition~\ref{d:growth-factor}
in a couple of special cases of the examples considered in~\S\ref{s:examples-1}. More precisely,
we study the threshold regularity condition $\Re\lambda>\max(r_u(m_u),r_s(m_s))$
given in~\eqref{e:threshold-mult}.

We start with the basic case when $\mathcal E=M\times\mathbb C$ is trivial,
$\mathbf X=X$, and $\varphi^t$ is volume preserving. In this case
the condition~\eqref{e:threshold-mult} becomes
\begin{equation}
\Re\lambda>\max(\theta_s m_u,-\theta_u m_s)
\end{equation}
where $\theta_s,\theta_u>0$ are the largest numbers such that
for each $\varepsilon>0$ there exists $C_\varepsilon>0$ such that
for all $t\geq 0$
$$
\|d\varphi^t|_{E_s}\|\leq C_\varepsilon e^{-(\theta_s-\varepsilon) t},\qquad
\|d\varphi^{-t}|_{E_u}\|\leq C_\varepsilon e^{-(\theta_u-\varepsilon) t}.
$$

We next discuss the case when $X$ is the generator of the geodesic flow
on an $n+1$-dimensional compact hyperbolic manifold~$(\Sigma,g)$
and $\mathbf X=\mathcal L_X$ acts on sections of the bundle of perpendicular differential $k$-forms
$\Omega^k_0$. In this case $\varphi^t$ is volume preserving,
$\dim E_u=\dim E_s=n$,
and for the correct choice of metric on $M$ (the Sasaki metric) we have
$$
|d\varphi^t(x)v|=\begin{cases}
|v|,& v\in E_0(x);\\
e^t|v|,& v\in E_u(x);\\
e^{-t}|v|, & v\in E_s(x).
\end{cases}
$$
It follows that the parallel transport $\mathscr T_{\mathbf X}^t(x)$
has norm $e^{\min(k,2n-k)t}$ for $t\geq 0$, and the condition~\eqref{e:threshold-mult} becomes
\begin{equation}
\Re\lambda>\max(m_u,-m_s)+\min(k,2n-k).
\end{equation}

\section{Meromorphic continuation}

In this section we state and prove the main result of this paper, Theorem~\ref{t:main}
(see~\S\ref{s:main}). 

\subsection{Anisotropic Sobolev spaces and statement of the result}
\label{s:anisotropic-spaces}

We first introduce the spaces on which meromorphic continuation holds.
We fix a function
$$
\mathfrak m\in C^\infty(T^*M\setminus 0;\mathbb R)
$$
which satisfies the following conditions:
\begin{itemize}
\item $\mathfrak m$ is positively homogeneous of degree~0, that is $\mathfrak m(x,\tau\xi)=\mathfrak m(x,\xi)$ for all $(x,\xi)\in T^*M\setminus 0$
and $\tau>0$;
\item there exist constants $m_u\leq 0\leq m_s$ such that
$m_u\leq \mathfrak m\leq m_s$ everywhere and
$$
\mathfrak m=m_u\quad\text{near }E_u^*\setminus 0,\qquad
\mathfrak m=m_s\quad\text{near }E_s^*\setminus 0
$$
where the dual unstable/stable spaces $E_u^*,E_s^*\subset T^*M$ were introduced
in~\eqref{e:stun-dual};
\item $H_p\mathfrak m\leq 0$ everywhere, where the vector field $H_p$
is introduced in~\S\ref{s:dynamics}; equivalently,
$\mathfrak m(\varphi^t(x),d\varphi^t(x)^{-T}\xi)\leq \mathfrak m(x,\xi)$
for all $(x,\xi)\in T^*M\setminus 0$ and $t\geq 0$.
\end{itemize}
Such $\mathfrak m$ exists for any choice of $m_u\leq 0\leq m_s$
by Lemma~\ref{l:weight}.

Given $\mathfrak m$, we fix a semiclassical pseudodifferential operator
$\mathbf F_{\mathfrak m}$ such that:
\begin{itemize}
\item $\mathbf F_{\mathfrak m}$ lies in $\Psi^{0+}_h(M;\End(\mathcal E)):=\bigcap_{\varepsilon>0}\Psi^\varepsilon_h(M;\End(\mathcal E))$
and $\mathbf F_{\mathfrak m}^*=\mathbf F_{\mathfrak m}$;
\item $\mathbf F_{\mathfrak m}$ is principally scalar and, for some fixed choice of Riemannian
metric on~$M$,
$$
\sigma_h(\mathbf F_{\mathfrak m})(x,\xi)=\mathfrak m(x,\xi)\log|\xi|\quad\text{when}\quad |\xi|\geq 1.
$$
\end{itemize}
For $t\geq 0$ we can define the exponential operators
\begin{equation}
  \label{e:expors}
e^{t\mathbf F_{\mathfrak m}}\in \Psi^{tm_s+}_h(M;\End(\mathcal E)),\qquad
e^{-t\mathbf F_{\mathfrak m}}\in \Psi^{-tm_u+}_h(M;\End(\mathcal E)).
\end{equation}
See~\cite[Theorem~8.6]{Zworski-Book} for the case of scalar operators and Weyl quantization on~$\mathbb R^n$
(with Beals's theorem for the Kohn--Nirenberg calculus given in~\cite[Theorem~9.12]{Zworski-Book}); the proof adapts to the case of manifolds and vector bundles studied here.
Alternatively, see~\cite[Appendix~A]{Faure-Roy-Sjostrand}.

We now define the semiclassical \emph{anisotropic Sobolev space} $H^{\mathfrak m}_h(M;\mathcal E)$ similarly to~\cite[\S8.3.1]{Zworski-Book}:
$$
H^{\mathfrak m}_h(M;\mathcal E):=e^{-\mathbf F_{\mathfrak m}}L^2(M;\mathcal E),\qquad
\|\mathbf u\|_{H^{\mathfrak m}_h}:=\|e^{\mathbf F_{\mathfrak m}}\mathbf u\|_{L^2}.
$$
The spaces $H^{\mathfrak m}_h(M;\mathcal E)$ for different values of~$h$
are all equivalent, with constants in the norm equivalency bounds depending on~$h$.
Therefore, we may use the notation $H^{\mathfrak m}(M;\mathcal E)$ when the choice of norm
is not important. We have
\begin{equation}
  \label{e:anis-squeeze}
H^{m_s}(M;\mathcal E)\ \subset\ H^{\mathfrak m}(M;\mathcal E)\ \subset\ H^{m_u}(M;\mathcal E)
\end{equation}
and the space $C^\infty(M;\mathcal E)$ is dense in $H^{\mathfrak m}(M;\mathcal E)$.

Fix open subsets
\begin{equation}
  \label{e:tilde-V}
\widetilde V_u,\widetilde V_s\subset \overline{T^*M}\setminus 0,\quad
\kappa(E_u^*)\subset\widetilde V_u,\quad
\kappa(E_s^*)\subset\widetilde V_s,
\end{equation}
such that $\mathfrak m=m_u$ on $\widetilde V_u$ and $\mathfrak m=m_s$ on $\widetilde V_s$.
Then the space $H^{\mathfrak m}_h(M;\mathcal E)$ is
equivalent to the usual Sobolev space $H^{m_u}_h(M;\mathcal E)$
microlocally on~$\widetilde V_u$, that is for
each $\mathbf A\in\Psi^0_h(M;\End(\mathcal E))$ with $\WFh(\mathbf A)\subset \widetilde V_u$,
there exists a constant $C$ such that
for each $\mathbf u\in C^\infty(M;\mathcal E)$ and each~$N$
\begin{equation}
  \label{e:anis-equiv}
\begin{aligned}
\|\mathbf A\mathbf u\|_{H^{\mathfrak m}_h}&\leq C\|\mathbf u\|_{H^{m_u}_h}+\mathcal O(h^\infty)\|\mathbf u\|_{H^{-N}_h},\\
\|\mathbf A\mathbf u\|_{H^{m_u}_h}&\leq C\|\mathbf u\|_{H^{\mathfrak m}_h}+\mathcal O(h^\infty)\|\mathbf u\|_{H^{-N}_h}.
\end{aligned}
\end{equation}
Similarly, $H^{\mathfrak m}_h(M;\mathcal E)$ is equivalent to the space
$H^{m_s}(M;\mathcal E)$ microlocally on~$\widetilde V_s$.

\subsubsection{Statement of the result}
\label{s:main}

We can now state the main result of this paper,
which gives meromorphic continuation of the Pollicott--Ruelle resolvent
on anisotropic Sobolev spaces to a specific half-plane:

\begin{theo}
  \label{t:main}
Let $X$ be the generator of an Anosov flow $\varphi^t$ on a compact
manifold~$M$, $\mathcal E$ be a smooth vector bundle over~$M$,
and $\mathbf X:C^\infty(M;\mathcal E)\to C^\infty(M;\mathcal E)$
be a lift of $X$ (see Definition~\ref{d:lift}).

Assume that the function $\mathfrak m\in C^\infty(T^*M\setminus 0;\mathbb R)$
satisfies the conditions in the beginning of~\S\ref{s:anisotropic-spaces},
for some constants $m_u\leq 0\leq m_s$. Let $H^{\mathfrak m}(M;\mathcal E)$
be the corresponding anisotropic Sobolev space.

Then the Pollicott--Ruelle resolvent $R_{\mathbf X}(\lambda)$
defined in~\eqref{e:P-R-resolvent-gen} admits a meromorphic continuation as a family
of operators $H^{\mathfrak m}(M;\mathcal E)\to H^{\mathfrak m}(M;\mathcal E)$
to the half-plane
\begin{equation}
  \label{e:threshold-thm}
\Re\lambda>\max(r_u(m_u),r_s(m_s))
\end{equation}
where $r_u(m_u),r_s(m_s)$ were introduced in Definition~\ref{d:growth-factor}.
\end{theo}
\Remark
By~\eqref{e:rus-big}, if we fix $\lambda$ then for $-m_u,m_s$
large enough the condition~\eqref{e:threshold-thm} holds.
Since $C^\infty(M;\mathcal E)\subset H^{\mathfrak m}(M;\mathcal E)\subset \mathcal D'(M;\mathcal E)$, we see that $R_{\mathbf X}(\lambda)$
continues meromorphically as a family of operators
$C^\infty(M;\mathcal E)\to\mathcal D'(M;\mathcal E)$ to $\lambda\in\mathbb C$.

\subsubsection{The conjugated operator}
\label{s:conjugated-operator}

The action of $\mathbf P=-ih\mathbf X$ on $H^{\mathfrak m}_h(M;\mathcal E)$ is equivalent to
the action on $L^2(M;\mathcal E)$ of the conjugated operator
\begin{equation}
  \label{e:tilde-P}
\widetilde{\mathbf P}:=e^{\mathbf F_{\mathfrak m}}\mathbf Pe^{-\mathbf F_{\mathfrak m}}.
\end{equation}
Using Taylor's formula with integral remainder
for the family of operators $e^{t\mathbf F_{\mathfrak m}}\mathbf Pe^{-t\mathbf F_{\mathfrak m}}$, $t\in [0,1]$, we see that for any $N\in\mathbb N$, we can expand $\widetilde{\mathbf P}$ as follows:
\begin{equation}
  \label{e:taylor}
\widetilde{\mathbf P}=\sum_{j=0}^{N-1} {\ad_{\mathbf F_{\mathfrak m}}^j \mathbf P\over j!}
+\int_0^1 (1-t)^{N-1}e^{t\mathbf F_{\mathfrak m}}{\ad_{\mathbf F_{\mathfrak m}}^N \mathbf P\over (N-1)!}e^{-t\mathbf F_{\mathfrak m}}\,dt
\end{equation}
where $\ad_{\mathbf F_{\mathfrak m}}\mathbf A=[\mathbf F_{\mathfrak m},\mathbf A]$
for any operator $\mathbf A$ on $C^\infty(M;\mathcal E)$.

Since
$\mathbf F_{\mathfrak m}\in\Psi^{0+}_h(M;\End(\mathcal E))$ is principally
scalar, we have $\ad_{\mathbf F_{\mathfrak m}}:\Psi^{m+}_h(M;\End(\mathcal E))\to
h\Psi^{m-1+}_h(M;\End(\mathcal E))$ for all~$m$.
Therefore, the $j$-th term in the sum in~\eqref{e:taylor} is
in $h^j\Psi^{1-j+}_h$; using~\eqref{e:expors}, we see
that the remainder is in $h^N\Psi^{1-N+m_s-m_u+}_h$.
Since $N$ can be chosen arbitrarily large, we in particular get the expansion
$$
\widetilde{\mathbf P}=\mathbf P+[\mathbf F_{\mathfrak m},\mathbf P]
+\mathcal O(h^2)_{\Psi^{-1+}_h(M;\End(\mathcal E))}.
$$
It follows that $\widetilde{\mathbf P}$ lies in $\Psi^1_h(M;\End(\mathcal E))$
and is principally scalar with
\begin{equation}
  \label{e:P-conj-symb}
\sigma_h(\widetilde{\mathbf P})=p+ih (H_p\mathfrak m)\log|\xi|,
\end{equation}
where we used that $H_p\log|\xi|\in S^0$ for $|\xi|\geq 1$.

An expansion of the form~\eqref{e:taylor} is valid for any pseudodifferential operator
in place of $\mathbf P$. In particular, we get
\begin{equation}
  \label{e:cutoff-conj}
\mathbf A\in\Psi^0_h(M;\End(\mathcal E))\ \Longrightarrow\
e^{\mathbf F_{\mathfrak m}}\mathbf A e^{-\mathbf F_{\mathfrak m}}\in\Psi^0_h(M;\End(\mathcal E))
\end{equation}
and the wavefrontset / elliptic set of $\mathbf A$ coincide
with those of~$e^{\mathbf F_{\mathfrak m}}\mathbf A e^{-\mathbf F_{\mathfrak m}}$.

\subsection{Invertibility of the perturbed operator}

We now state the key estimate for the proof of Theorem~\ref{t:main},
which gives invertibility for the operator $\mathbf P=-ih\mathbf X$
on the anisotropic Sobolev space $H^{\mathfrak m}_h(M;\mathcal E)$ when
modified by a complex absorbing operator. Consider
the dual space of $H^{\mathfrak m}_h(M;\mathcal E)$ (with respect to the $L^2$ inner product), given by
$$
H^{-\mathfrak m}_h(M;\mathcal E):=e^{\mathbf F_{\mathfrak m}}L^2(M;\mathcal E).
$$
Fix a principally scalar pseudodifferential operator
$$
\mathbf Q\in\Psi^0_h(M;\End(\mathcal E)),\quad
\sigma_h(\mathbf Q)\geq 0
$$
such that $\WFh(\mathbf Q)$ does not intersect the fiber infinity $\partial\overline{T^*M}$
and the elliptic set $\Ell_h(\mathbf Q)$ contains the zero section of $T^*M$. For technical
reasons we also assume that
\begin{equation}
  \label{e:Q-tilde-V}
\WFh(\mathbf Q)\cap \widetilde V_u=\WFh(\mathbf Q)\cap\widetilde V_s=\emptyset
\end{equation}
where $\widetilde V_u,\widetilde V_s\subset \overline{T^*M}\setminus 0$
were introduced in~\eqref{e:tilde-V}.
\begin{lemm}\label{l:invertibility}
Let $\mathfrak m$ satisfy the conditions in the beginning of~\S\ref{s:anisotropic-spaces}
and assume that $\Omega\subset\mathbb C$ is a compact set such that
\begin{equation}
  \label{e:inv-threshold}
\Re\lambda>\max(r_u(m_u),r_s(m_s))\quad\text{for all}\quad
\lambda\in\Omega.
\end{equation}
Then we have the following
estimates for $h$ small enough, all $\lambda\in\Omega$, and all $\mathbf u\in C^\infty(M;\mathcal E)$,
with the constants independent of~$h,\lambda,\mathbf u$:
\begin{align}
  \label{e:invest-1}
\|\mathbf u\|_{H^{\mathfrak m}_h}&\leq Ch^{-1}\|(\mathbf P-ih\lambda-i\mathbf Q)\mathbf u\|_{H^{\mathfrak m}_h},\\
  \label{e:invest-2}
\|\mathbf u\|_{H^{-\mathfrak m}_h}&\leq Ch^{-1}\|(\mathbf P-ih\lambda-i\mathbf Q)^*\mathbf u\|_{H^{-\mathfrak m}_h}.
\end{align}
\end{lemm}
We will only give the proof of the direct estimate~\eqref{e:invest-1}.
The adjoint estimate~\eqref{e:invest-2} follows from
the direct estimate for the operator $\mathbf X^*$ which is a lift of the vector
field~$-X$. Note that $-(\mathbf P-ih\lambda-i\mathbf Q)^*=-ih\mathbf X^*-ih\bar\lambda -i\mathbf Q^*$. The associated flow is $\varphi^{-t}$ and the stable/unstable
spaces are switched places. The constants $m_u,m_s$
are replaced by $-m_s,-m_u$ and
the weight~$\mathfrak m$ is replaced by $-\mathfrak m$.
Using~\eqref{e:growth-factor}, we see that the threshold condition~\eqref{e:inv-threshold} gives the analogous condition
for the operator $\mathbf X^*$.
(Here the parallel transport corresponding to~$\mathbf X^*$ can be computed using~\eqref{e:evolution-parallel}, as $(e^{-t\mathbf X})^*=e^{-t\mathbf X^*}$.)

The proof of~\eqref{e:invest-1} is broken into several components.
Throughout this section we assume that $h$ is small, $\lambda\in\Omega$,
and $\mathbf u\in C^\infty(M;\mathcal E)$. The constants in the estimates
below are independent of~$h$, and the Sobolev exponent $N$ in the remainders
can be chosen arbitrarily.

\subsubsection{Elliptic estimate}
\label{s:elliptic}

We first state the elliptic estimate:
\begin{lemm}
  \label{l:our-elliptic}
Assume that $\mathbf A\in\Psi^0_h(M;\End(\mathcal E))$ and
$$
\WFh(\mathbf A)\ \subset\ \Ell_h(\mathbf P)\cup\Ell_h(\mathbf Q).
$$ 
Then
\begin{equation}
  \label{e:our-elliptic}
\|\mathbf A \mathbf u\|_{H^{\mathfrak m}_h}\leq C\|(\mathbf P-ih\lambda-i\mathbf Q)\mathbf u\|_{H^{\mathfrak m}_h}+\mathcal O(h^\infty)\|\mathbf u\|_{H^{-N}_h}.
\end{equation}
\end{lemm}
To prove Lemma~\ref{l:our-elliptic}, we first reduce it to an estimate in
the space $L^2$ for the conjugated operator
$\widetilde{\mathbf P}-ih\lambda-\widetilde{\mathbf Q}$ where
\begin{equation}
  \label{e:Q-conjugated}
\widetilde{\mathbf Q}:=e^{\mathbf F_{\mathfrak m}}\mathbf Qe^{-\mathbf F_{\mathfrak m}}.
\end{equation}
Denote $\widetilde{\mathbf A}:=e^{\mathbf F_{\mathfrak m}}\mathbf Ae^{-\mathbf F_{\mathfrak m}}$. Then~\eqref{e:our-elliptic} follows from the estimate
\begin{equation}
  \label{e:our-elliptic-2}
\|\widetilde{\mathbf A}\mathbf v\|_{L^2}\leq C\|(\widetilde{\mathbf P}-ih\lambda-i\widetilde{\mathbf Q})\mathbf v\|_{L^2}+\mathcal O(h^\infty)\|\mathbf v\|_{H^{-N}_h}
\end{equation}
where we put $\mathbf v:=e^{\mathbf F_{\mathfrak m}}\mathbf u\in C^\infty(M;\mathcal E)$.

Since $\WFh(\mathbf Q)$ does not intersect the fiber infinity $\partial\overline{T^*M}$,
using the expansion~\eqref{e:taylor} for $\mathbf Q$ in place of $\mathbf P$ we see
that $\widetilde{\mathbf Q}=\mathbf Q+\mathcal O(h)_{\Psi^{-\infty}_h(M;\End(\mathcal E))}$.
Moreover,
by~\eqref{e:cutoff-conj} the operator $\widetilde{\mathbf A}\in\Psi^0_h(M;\End(\mathcal E))$
has the same wavefront set as $\mathbf A$. It follows
that $\WFh(\widetilde{\mathbf A})\subset \Ell_h(\widetilde{\mathbf P}-ih\lambda-i\widetilde{\mathbf Q})$.
Now~\eqref{e:our-elliptic-2}
follows from the standard elliptic estimate~\cite[Theorem~E.33]{DZ-Book}
whose proof adapts directly to the case of operators on vector bundles.

\subsubsection{Propagation of singularities}
\label{s:propagation}

Our next estimate is propagation of singularities:
\begin{lemm}
  \label{l:our-propagation}
Assume that $\mathbf A,\mathbf B,\mathbf B_1\in\Psi^0_h(M;\End(\mathcal E))$ and
the following control condition holds:
$$
\begin{gathered}
\text{for all }(x,\xi)\in \WFh(\mathbf A)\quad\text{there exists }T\geq 0\quad\text{such that}\\
e^{-TH_p}(x,\xi)\in\Ell_h(\mathbf B)\quad\text{and}\quad
e^{-tH_p}(x,\xi)\in\Ell_h(\mathbf B_1)\quad\text{for all }t\in [0,T].
\end{gathered}
$$
Then 
$$
\|\mathbf A\mathbf u\|_{H^{\mathfrak m}_h}
\leq C\|\mathbf B\mathbf u\|_{H^{\mathfrak m}_h}
+Ch^{-1}\|\mathbf B_1(\mathbf P-ih\lambda-i\mathbf Q)\mathbf u\|_{H^{\mathfrak m}_h}
+\mathcal O(h^\infty)\|\mathbf u\|_{H^{-N}_h}.
$$
\end{lemm}
Similarly to~\S\ref{s:elliptic}, Lemma~\ref{l:our-propagation} can be reduced
to an estimate in the space $L^2$ for the conjugated operator
$\widetilde{\mathbf P}-ih\lambda-i\widetilde{\mathbf Q}$. The latter estimate
is proved using the same positive commutator estimate as standard scalar propagation of singularities~\cite[Theorem~E.47]{DZ-Book}, using
a principally scalar multiplier~$\mathbf G$, given that:
\begin{itemize}
\item $\widetilde{\mathbf P}-ih\lambda-i\widetilde{\mathbf Q}\in\Psi^1_h(M;\End(\mathcal E))$
is principally scalar;
\item $\Re\sigma_h(\widetilde{\mathbf P}-ih\lambda-i\widetilde{\mathbf Q})=p$;
\item $\Im\sigma_h(\widetilde{\mathbf P}-ih\lambda-i\widetilde{\mathbf Q})\leq 0$.
Indeed, $\Im\sigma_h(\widetilde{\mathbf P})=h(H_p\mathfrak m)\log|\xi|$ by~\eqref{e:P-conj-symb} and $H_p\mathfrak m\leq 0$ as required in the beginning of~\S\ref{s:anisotropic-spaces}.
Moreover, $\sigma_h(\widetilde{\mathbf Q})=\sigma_h(\mathbf Q)\geq 0$;
\item the sharp G\r arding inequality applies to principally scalar operators
in $\Psi^{2m}_h(M;\End(\mathcal E))$ with nonnegative principal symbol,
as follows for example from Lemma~\ref{l:sharp-Garding}.
\end{itemize}

\subsubsection{Radial estimates}
\label{s:radial}

We now prove the two radial estimates that are crucial in the proof of Lemma~\ref{l:invertibility}. This is the place in the argument where the threshold
regularity condition~\eqref{e:inv-threshold} is important.
Recall the
sets $\widetilde V_u,\widetilde V_s\subset\overline{T^*M}$ introduced in~\eqref{e:tilde-V}.

We start with the high regularity radial estimate at the set
$\kappa(E_s^*)\subset \partial\overline{T^*M}$. 
\begin{lemm}
  \label{l:radial-high}
There exist operators
$$
\mathbf A_s,\mathbf B_{1,s}\in\Psi^0_h(M;\End(\mathcal E)),\quad
\kappa(E_s^*)\subset\Ell_h(\mathbf A_s),\quad
\WFh(\mathbf A_s)\cup\WFh(\mathbf B_{1,s})\subset \widetilde V_s
$$
such that the following estimate holds:
\begin{equation}
  \label{e:radial-high}
\|\mathbf A_s\mathbf u\|_{H^{\mathfrak m}_h}
\leq Ch^{-1}\|\mathbf B_{1,s}(\mathbf P-ih\lambda-i\mathbf Q)\mathbf u\|_{H^{\mathfrak m}_h}
+\mathcal O(h^\infty)\|\mathbf u\|_{H^{-N}_h}.
\end{equation}
\end{lemm}
\begin{proof}
1. Since $H^{\mathfrak m}_h$ is equivalent to $H^{m_s}_h$ microlocally
on $\widetilde V_s$ (see~\eqref{e:anis-equiv}) and
$\WFh(\mathbf Q)\cap\widetilde V_s=\emptyset$ (see~\eqref{e:Q-tilde-V}),
it suffices to show the estimate
\begin{equation}
  \label{e:radial-high-2}
\|\mathbf A_s\mathbf u\|_{H^{m_s}_h}
\leq Ch^{-1}\|\mathbf B_{1,s}(\mathbf P-ih\lambda)\mathbf u\|_{H^{m_s}_h}
+\mathcal O(h^\infty)\|\mathbf u\|_{H^{-N}_h}.
\end{equation}

\noindent 2. We now follow the proof of~\cite[Theorem~E.52]{DZ-Book}, indicating the necessary changes.
Since the threshold condition~\eqref{e:inv-threshold} holds,
Lemma~\ref{l:multipliers} applies to give a section $\mathbf w_s\in C^\infty(T^*M\setminus 0;\End(\pi^*\mathcal E))$ which is positive definite everywhere, positively homogeneous of
degree~$2m_s$, and satisfies
(where `$\boldsymbol<$\,0' means `negative definite')
\begin{equation}
  \label{e:radhigh-1}
(H_{\mathbf X}-2\Re\lambda)\mathbf w_s(x,\xi)\boldsymbol< 0\quad\text{for all}\quad
\lambda\in\Omega,\quad
(x,\xi)\in E_s^*\setminus 0.
\end{equation}
Fix an open set $U_s\subset \widetilde V_s$ such that
$\kappa(E_s^*)\subset U_s$ and there exists $\delta>0$ such that
\begin{equation}
  \label{e:radhigh-2}
(\textstyle{1\over 2}H_{\mathbf X}-\Re\lambda+\delta)\mathbf w_s(x,\xi)\boldsymbol< 0\quad\text{for all}\quad
\lambda\in \Omega,\quad
(x,\xi)\in U_s.
\end{equation}
Arguing as in the proof of Lemma~\ref{l:weight}
(see also~\cite[Lemma~E.53]{DZ-Book}), we construct a function
\begin{equation}
  \label{e:radhigh-3}
\chi_s\in \CIc(U_s;[0,1]),\qquad
\chi_s=1\quad\text{near}\quad \kappa(E_s^*),\qquad
H_p\chi_s\leq 0.
\end{equation}
Denote by $\sqrt{\mathbf w_s}$ the square root of $\mathbf w_s$, which
is a positive definite section in $C^\infty(T^*M\setminus 0;\End(\pi^*\mathcal E))$
and positively homogeneous of degree~$m_s$.
Define
$$
\mathbf g_s:=\chi_s \sqrt{\mathbf w_s}\ \in\ C^\infty(T^*M;\End(\pi^*\mathcal E))
$$
and note that $\mathbf g_s$ lies in the symbol class $S^{m_s}$.

\noindent 3. Take a pseudodifferential operator
$$
\mathbf G_s\in\Psi^{m_s}_h(M;\End(\mathcal E)),\quad
\WFh(\mathbf G_s)\subset U_s,\quad
\sigma_h(\mathbf G_s)=\mathbf g_s.
$$
Note that $\mathbf G_s$ is elliptic on $\kappa(E_s^*)$.
Fix also operators
$\mathbf A_s,\mathbf B_{2,s}\in\Psi^0_h(M;\End(\mathcal E))$ such that
$$
\kappa(E_s^*)\subset\Ell_h(\mathbf A_s),\quad
\WFh(\mathbf A_s)\subset\Ell_h(\mathbf G_s),\quad
\WFh(\mathbf G_s)\subset \Ell_h(\mathbf B_{2,s}),\quad
\WFh(\mathbf B_{2,s})\subset U_s.
$$
By Lemma~\ref{l:commutator-computed}, we have
\begin{equation}
  \label{e:radhigh-3.5}
h^{-1}\Im\langle (\mathbf P-ih\lambda)\mathbf u,\mathbf G_s^*\mathbf G_s\mathbf u\rangle_{L^2}
+\delta\|\mathbf G_s\mathbf u\|_{L^2}^2
=\langle\mathbf Z_s \mathbf u,\mathbf u\rangle_{L^2}
\end{equation}
where
$$
\mathbf Z_s\in \Psi^{2m_s}(M;\End(\mathcal E)),\quad
\mathbf Z_s^*=\mathbf Z_s,\quad
\WFh(\mathbf Z_s)\subset \Ell_h(\mathbf B_{2,s})
$$
has principal symbol
\begin{equation}
  \label{e:radhigh-4}
\sigma_h(\mathbf Z_s)=(\textstyle{1\over 2}H_{\mathbf X}-\Re\lambda+\delta)(\chi_s^2\mathbf w_s)
=\chi_s(H_p\chi_s)\mathbf w_s+\chi_s^2(\textstyle{1\over 2}H_{\mathbf X}-\Re\lambda+\delta)\mathbf w_s.
\end{equation}
By~\eqref{e:radhigh-2}--\eqref{e:radhigh-3},
each of the two summands on the right-hand side of~\eqref{e:radhigh-4}
is the product of a nonnegative function in $\CIc(U_s)$ and a self-adjoint section of $\End(\pi^*\mathcal E)$ which is positively homogeneous of degree $2m_s$ and negative definite on $U_s$.
Thus the version of the sharp G\r arding inequality given in Lemma~\ref{l:sharp-Garding}
gives
$$
\langle\mathbf Z_s\mathbf u,\mathbf u\rangle_{L^2}\leq Ch\|\mathbf B_{2,s}\mathbf u\|_{H^{m_s-{1\over 2}}_h}^2
+\mathcal O(h^\infty)\|\mathbf u\|_{H^{-N}_h}^2.
$$
Together with~\eqref{e:radhigh-3.5} this implies
$$
\|\mathbf G_s \mathbf u\|_{L^2}^2\leq Ch^{-1}\|\mathbf B_{2,s}(\mathbf P-ih\lambda)\mathbf u\|_{H^{m_s}_h}\cdot\|\mathbf G_s\mathbf u\|_{L^2}+Ch\|\mathbf B_{2,s}\mathbf u\|_{H^{m_s-{1\over 2}}_h}^2
+\mathcal O(h^\infty)\|\mathbf u\|_{H^{-N}_h}^2
$$
which gives the estimate
\begin{equation}
  \label{e:radhigh-5}
\|\mathbf G_s\mathbf u\|_{L^2}\leq Ch^{-1}\|\mathbf B_{2,s}(\mathbf P-ih\lambda)\mathbf u\|_{H^{m_s}_h}
+Ch^{1\over 2}\|\mathbf B_{2,s}\mathbf u\|_{H^{m_s-{1\over 2}}_h}+\mathcal O(h^\infty)\|\mathbf u\|_{H^{-N}_h}.
\end{equation}

\noindent 4.
We now argue similarly to step~2 of the proof of~\cite[Theorem~E.52]{DZ-Book}.
By the elliptic estimate we can replace $\|\mathbf G_s\mathbf u\|_{L^2}$ on the left-hand side of~\eqref{e:radhigh-5} by $\|\mathbf A_s\mathbf u\|_{H^{m_s}_h}$.
If the set~$U_s$ is chosen small enough, then
propagation of singularities gives
\begin{equation}
  \label{e:radhigh-6}
\|\mathbf B_{2,s}\mathbf u\|_{H^{m_s-{1\over 2}}_h}\leq
C \|\mathbf A_s\mathbf u\|_{H^{m_s-{1\over 2}}_h}
+Ch^{-1}\|\mathbf B_{1,s}(\mathbf P-ih\lambda)\mathbf u\|_{H^{m_s}_h}
+\mathcal O(h^\infty)\|\mathbf u\|_{H^{-N}_h}
\end{equation}
for some $\mathbf B_{1,s}\in\Psi^0_h(M;\End(\mathcal E))$ such that
$$
\WFh(\mathbf B_{2,s})\subset \Ell_h(\mathbf B_{1,s}),\quad
\WFh(\mathbf B_{1,s})\subset \widetilde V_s.
$$
Combining~\eqref{e:radhigh-5}
and~\eqref{e:radhigh-6} and taking $h$ small enough,
we get~\eqref{e:radial-high-2}.
\end{proof}
We next give the low regularity radial estimate at the set $\kappa(E_u^*)$:
\begin{lemm}
  \label{l:radial-low}
There exist operators
$$
\begin{gathered}
\mathbf A_u,\mathbf B_u,\mathbf B_{1,u}\in\Psi^0_h(M;\End(\mathcal E)),\quad
\kappa(E_u^*)\subset\Ell_h(\mathbf A_u),\\
\WFh(\mathbf A_u)\cup\WFh(\mathbf B_{1,u})\subset \widetilde V_u,\quad
\WFh(\mathbf B_u)\subset \widetilde V_u\setminus \kappa(E_u^*)
\end{gathered}
$$
such that the following estimate holds:
\begin{equation}
  \label{e:radial-low}
\|\mathbf A_u\mathbf u\|_{H^{\mathfrak m}_h}
\leq C\|\mathbf B_u\mathbf u\|_{H^{\mathfrak m}_h}+
Ch^{-1}\|\mathbf B_{1,u}(\mathbf P-ih\lambda-i\mathbf Q)\mathbf u\|_{H^{\mathfrak m}_h}
+\mathcal O(h^\infty)\|\mathbf u\|_{H^{-N}_h}.
\end{equation}
\end{lemm}
\begin{proof}
1. We argue similarly to the proof of~\cite[Theorem~E.54]{DZ-Book}, making
changes similar to the proof of Lemma~\ref{l:radial-high}.
Since $H^{\mathfrak m}_h$ is equivalent
to $H^{m_u}_h$ microlocally on $\widetilde V_u$,
it suffices to show the estimate
\begin{equation}
  \label{e:radial-low-2}
\|\mathbf A_u\mathbf u\|_{H^{m_u}_h}
\leq C\|\mathbf B_u\mathbf u\|_{H^{m_u}_h}+
Ch^{-1}\|\mathbf B_{1,u}(\mathbf P-ih\lambda)\mathbf u\|_{H^{m_u}_h}
+\mathcal O(h^\infty)\|\mathbf u\|_{H^{-N}_h}.
\end{equation}

\noindent 2. Since the threshold condition~\eqref{e:inv-threshold} holds,
Lemma~\ref{l:multipliers} applies to give a section
$\mathbf w_u\in C^\infty(T^*M\setminus 0;\End(\pi^*\mathcal E))$
which is positive definite everywhere, positively homogeneous of degree~$2m_u$,
and satisfies
$$
(H_{\mathbf X}-2\Re\lambda)\mathbf w_u(x,\xi)\boldsymbol<0\quad\text{for all}\quad
\lambda\in\Omega,\quad
(x,\xi)\in E_u^*\setminus 0.
$$
Fix an open set $U_u\subset\widetilde V_u$ such that $\kappa(E_u^*)\subset U_u$
and there exists $\delta>0$ such that
\begin{equation}
  \label{e:radlow-1}
(\textstyle{1\over 2}H_{\mathbf X}-\Re\lambda+\delta)\mathbf w_u(x,\xi)\boldsymbol <0
\quad\text{for all}\quad \lambda\in\Omega,\quad (x,\xi)\in U_u.
\end{equation}
Take an arbitrary cutoff
$$
\chi_u\in \CIc(U_u;[0,1]),\quad
\chi_u=1\quad\text{near}\quad \kappa(E_u^*)
$$
and define
$$
\mathbf g_u:=\chi_u\sqrt{\mathbf w_u}\ \in\ C^\infty(T^*M;\End(\pi^*\mathcal E))
$$
which lies in the class $S^{m_u}$.

\noindent 3. Take a pseudodifferential operator
$$
\mathbf G_u\in \Psi^{m_u}_h(M;\End(\mathcal E)),\quad
\WFh(\mathbf G_u)\subset U_u,\quad
\sigma_h(\mathbf G_u)=\mathbf g_u.
$$
Note that $\mathbf G_u$ is elliptic on $\kappa(E_u^*)$.
Fix a cutoff function
\begin{equation}
  \label{e:radlow-2}
\psi_u\in \CIc(U_u\setminus\kappa(E_u^*))\quad\text{such that}\quad
\chi_u(H_p\chi_u)\leq|\psi_u|^2\quad\text{everywhere}
\end{equation}
and an operator $\mathbf E_u\in\Psi^{m_u}_h(M;\End(\mathcal E))$ such that
$$
\WFh(\mathbf E_u)\subset U_u\setminus \kappa(E_u^*),\quad
\sigma_h(\mathbf E_u)=\psi_u\sqrt{\mathbf w_u}.
$$
Now, fix $\mathbf A_u,\mathbf B_u\in\Psi^0_h(M;\End(\mathcal E))$ such that,
putting $\mathbf B_{1,u}:=\mathbf A_u^*\mathbf A_u+\mathbf B_u^*\mathbf B_u$,
$$
\begin{aligned}
\kappa(E_u^*)\subset\Ell_h(\mathbf A_u)&,\quad
\WFh(\mathbf A_u)\subset\Ell_h(\mathbf G_u),\quad
\WFh(\mathbf E_u)\subset\Ell_h(\mathbf B_u),\\
\WFh(\mathbf B_u)\subset U_u\setminus \kappa(E_u^*),&\quad
\WFh(\mathbf G_u)\subset\Ell_h(\mathbf A_u)\cup\Ell_h(\mathbf B_u)\subset\Ell_h(\mathbf B_{1,u}).
\end{aligned}
$$
By Lemma~\ref{l:commutator-computed}, we have
\begin{equation}
  \label{e:radlow-2.5}
h^{-1}\Im\langle (\mathbf P-ih\lambda)\mathbf u,\mathbf G_u^*\mathbf G_u\mathbf u\rangle_{L^2}
-\|\mathbf E_u\mathbf u\|_{L^2}^2+\delta\|\mathbf G_u\mathbf u\|_{L^2}^2=\langle \mathbf Z_u\mathbf u,\mathbf u\rangle_{L^2}
\end{equation}
where
$$
\mathbf Z_u\in\Psi^{2m_u}(M;\End(\mathcal E)),\quad
\mathbf Z_u^*=\mathbf Z_u,\quad
\WFh(\mathbf Z_u)\subset\Ell_h(\mathbf B_{1,u})
$$
has principal symbol
\begin{equation}
  \label{e:radlow-3}
\sigma_h(\mathbf Z_u)=(\chi_u(H_p\chi_u)-|\psi_u|^2)\mathbf w_u
+\chi_u^2(\textstyle{1\over 2}H_{\mathbf X}-\Re\lambda+\delta)\mathbf w_u.
\end{equation}
By~\eqref{e:radlow-1}--\eqref{e:radlow-2}, each of the two summands on the right-hand
side of~\eqref{e:radlow-3} is the product of a nonnegative function
in $\CIc(U_u)$ and a self-adjoint section of $\End(\mathcal E)$ which is positively homogeneous of degree $2m_u$ and negative
definite on $U_u$. Thus Lemma~\ref{l:sharp-Garding} gives
$$
\langle \mathbf Z_u\mathbf u,\mathbf u\rangle_{L^2}\leq Ch\|\mathbf B_{1,u}\mathbf u\|_{H^{m_u-{1\over 2}}_h}^2+\mathcal O(h^\infty)\|\mathbf u\|_{H^{-N}_h}^2
$$
which together with~\eqref{e:radlow-2.5} implies
\begin{equation}
  \label{e:radlow-4}
\begin{aligned}
\|\mathbf G_u\mathbf u\|_{L^2}\leq\,&C\|\mathbf E_u\mathbf u\|_{L^2}
+Ch^{-1}\|\mathbf B_{1,u}(\mathbf P-ih\lambda)\mathbf u\|_{H^{m_u}_h}\\
&+Ch^{1\over 2}\|\mathbf B_{1,u}\mathbf u\|_{H^{m_u-{1\over 2}}_h}+\mathcal O(h^\infty)\|\mathbf u\|_{H^{-N}_h}.
\end{aligned}
\end{equation}

\noindent 4. By the elliptic estimate, we can replace $\|\mathbf G_u\mathbf u\|_{L^2}$
on the left-hand side of~\eqref{e:radlow-4} by~$\|\mathbf A_u\mathbf u\|_{H^{m_u}_h}$.
Similarly we may replace $\|\mathbf E_u\mathbf u\|_{L^2}$
on the right-hand side of~\eqref{e:radlow-4} by~$\|\mathbf B_u\mathbf u\|_{H^{m_u}_h}$.
Finally, recalling the definition of $\mathbf B_{1,u}$ we see that
$$
\|\mathbf B_{1,u}\mathbf u\|_{H^{m_u-{1\over 2}}_h}
\leq C\big(\|\mathbf A_u\mathbf u\|_{H^{m_u-{1\over 2}}_h}
+\|\mathbf B_u\mathbf u\|_{H^{m_u-{1\over 2}}_h}\big).
$$
Taking $h$ small enough in~\eqref{e:radlow-4}, we now obtain~\eqref{e:radial-low-2}.
\end{proof}

\subsubsection{Proof of Lemma~\ref{l:invertibility}}

We are now ready to finish the proof of Lemma~\ref{l:invertibility},
following the proof of~\cite[Proposition~3.4]{DZ-zeta}.
Let $\mathbf A_s,\mathbf A_u,\mathbf B_u\in\Psi^0_h(M;\End(\mathcal E))$
be the operators from Lemmas~\ref{l:radial-high}--\ref{l:radial-low}.
We first combine ellipticity, propagation of singularities, and the high regularity
radial estimate to get
\begin{lemm}\label{l:almost-there}
Let $\mathbf A\in\Psi^0_h(M;\End(\mathcal E))$ satisfy
$\WFh(\mathbf A)\cap \kappa(E_u^*)=\emptyset$. Then
\begin{equation}
\label{e:almost-there}
\|\mathbf A\mathbf u\|_{H^{\mathfrak m}_h}\leq
Ch^{-1}\|(\mathbf P-ih\lambda-i\mathbf Q)\mathbf u\|_{H^{\mathfrak m}_h}
+\mathcal O(h^\infty)\|\mathbf u\|_{H^{-N}_h}.
\end{equation}
\end{lemm}
\begin{proof}
Fix an operator
$\widetilde{\mathbf Q}\in\Psi^0_h(M;\End(\mathcal E))$
such that $\WFh(\widetilde{\mathbf Q})\subset\Ell_h(\mathbf Q)$
and $\Ell_h(\widetilde{\mathbf Q})$ contains the zero section of~$T^*M$.
Define the open set $\mathcal U\subset\overline{T^*M}$ as follows:
$$
\mathcal U:=\{(x,\xi)\in\overline{T^*M}\mid\exists t\geq 0:\ 
e^{-tH_p}(x,\xi)\in\Ell_h(\widetilde{\mathbf Q})\cup\Ell_h(\mathbf A_s)\}.
$$
Since $\Ell_h(\mathbf A_s)$ contains $\kappa(E_s^*)$,
by Lemma~\ref{l:flow-char-global}
we have $\WFh(\mathbf A)\cap\{p=0\}\subset\mathcal U$,
that is $\WFh(\mathbf A)\subset\Ell_h(\mathbf P)\cup\mathcal U$.
Using a microlocal partition of unity, we write
$$
\mathbf A=\mathbf A_1+\mathbf A_2,\quad
\mathbf A_1,\mathbf A_2\in\Psi^0_h(M;\End(\mathcal E)),\quad
\WFh(\mathbf A_1)\subset \Ell_h(\mathbf P),\quad
\WFh(\mathbf A_2)\subset \mathcal U.
$$
By the elliptic estimate, Lemma~\ref{l:our-elliptic}, we have
\begin{equation}
  \label{e:almost-1}
\|\mathbf A_1 \mathbf u\|_{H^{\mathfrak m}_h}
+\|\widetilde{\mathbf Q}\mathbf u\|_{H^{\mathfrak m}_h}
\leq C\|(\mathbf P-ih\lambda-i\mathbf Q)\mathbf u\|_{H^{\mathfrak m}_h}
+\mathcal O(h^\infty)\|\mathbf u\|_{H^{-N}_h}.
\end{equation}
Next, by propagation of singularities, Lemma~\ref{l:our-propagation},
with $\mathbf B:=\mathbf A_s+\widetilde{\mathbf Q}$, $\Ell_h(\mathbf B)=\Ell_h(\mathbf A_s)\cup\Ell_h(\widetilde{\mathbf Q})$, we have
\begin{equation}
  \label{e:almost-2}
\begin{aligned}
\|\mathbf A_2\mathbf u\|_{H^{\mathfrak m}_h}\leq&\, C\|\mathbf A_s\mathbf u\|_{H^{\mathfrak m}_h}
+C\|\widetilde{\mathbf Q}\mathbf u\|_{H^{\mathfrak m}_h}
\\&+Ch^{-1}\|(\mathbf P-ih\lambda-i\mathbf Q)\mathbf u\|_{H^{\mathfrak m}_h}
+\mathcal O(h^\infty)\|\mathbf u\|_{H^{-N}_h}.
\end{aligned}
\end{equation}
Finally, recall that by the high regularity radial estimate, Lemma~\ref{l:radial-high},
\begin{equation}
  \label{e:almost-3}
\|\mathbf A_s\mathbf u\|_{H^{\mathfrak m}_h}\leq Ch^{-1}\|(\mathbf P-ih\lambda-i\mathbf Q)\mathbf u\|_{H^{\mathfrak m}_h}+\mathcal O(h^\infty)\|\mathbf u\|_{H^{-N}_h}.
\end{equation}
Putting together~\eqref{e:almost-1}--\eqref{e:almost-3}, we get~\eqref{e:almost-there}.
\end{proof}
Now, recall that the low regularity radial estimate, Lemma~\ref{l:radial-low}, gives
\begin{equation}
  \label{e:lofi}
\|\mathbf A_u\mathbf u\|_{H^{\mathfrak m}_h}\leq C\|\mathbf B_u\mathbf u\|_{H^{\mathfrak m}_h}
+Ch^{-1}\|(\mathbf P-ih\lambda-i\mathbf Q)\mathbf u\|_{H^{\mathfrak m}_h}
+\mathcal O(h^\infty)\|\mathbf u\|_{H^{-N}_h}.
\end{equation}
Take $\mathbf A\in \Psi^0_h(\mathbf M;\End(\mathcal E))$ such that
$$
\overline{T^*M}\setminus\Ell_h(\mathbf A_u)\subset\Ell_h(\mathbf A),\quad
\WFh(\mathbf A)\subset\overline{T^*M}\setminus\kappa(E_u^*).
$$
Since $\WFh(\mathbf B_u)\cap \kappa(E_u^*)=\emptyset$, Lemma~\ref{l:almost-there} applies
to both $\mathbf A$ and $\mathbf B_u$ to give
\begin{equation}
  \label{e:altoid}
\|\mathbf A\mathbf u\|_{H^{\mathfrak m}_h}+
\|\mathbf B_u\mathbf u\|_{H^{\mathfrak m}_h}
\leq Ch^{-1}\|(\mathbf P-ih\lambda-i\mathbf Q)\mathbf u\|_{H^{\mathfrak m}_h}
+\mathcal O(h^\infty)\|\mathbf u\|_{H^{-N}_h}.
\end{equation}
Since $\mathbf A^*\mathbf A+\mathbf A_u^*\mathbf A_u\in \Psi^0_h(M;\End(\mathcal E))$
is elliptic on the entire $\overline{T^*M}$, we can use the elliptic estimate
to derive from~\eqref{e:altoid} and~\eqref{e:lofi} the bound
$$
\begin{aligned}
\|\mathbf u\|_{H^{\mathfrak m}_h}&\leq
C\|\mathbf A\mathbf u\|_{H^{\mathfrak m}_h}+C\|\mathbf A_u\mathbf u\|_{H^{\mathfrak m}_h}
+\mathcal O(h^\infty)\|u\|_{H^{-N}_h}\\
&\leq Ch^{-1}\|(\mathbf P-ih\lambda-i\mathbf Q)\mathbf u\|_{H^{\mathfrak m}_h}
+\mathcal O(h^\infty)\|\mathbf u\|_{H^{-N}_h}.
\end{aligned}
$$
For $h$ small enough, we may remove the last term on the right-hand side, obtaining~\eqref{e:invest-1} and finishing the proof of Lemma~\ref{l:invertibility}.

\subsection{Meromorphic continuation}

We finally give the proof of Theorem~\ref{t:main}, following~\cite[\S\S3.3--3.4]{DZ-zeta}.
Let $\Omega\subset\mathbb C$ be a compact set satisfying the threshold regularity condition~\eqref{e:inv-threshold}.
Fix $h>0$ small enough so that Lemma~\ref{l:invertibility} applies.
We henceforth suppress the subscript $h$ in the notation $H^{\mathfrak m}_h$.
Define the space
$$
D^{\mathfrak m}:=\{\mathbf u\in H^{\mathfrak m}\mid \mathbf P\mathbf u\in H^{\mathfrak m}\}
$$
with the Hilbert norm
$$
\|\mathbf u\|_{D^{\mathfrak m}}^2:=\|\mathbf u\|_{H^{\mathfrak m}}^2
+\|\mathbf P\mathbf u\|_{H^{\mathfrak m}}^2.
$$
Since $\WFh(\mathbf Q)$ does not intersect the fiber infinity, $\mathbf Q$
is a smoothing operator. In particular, $\mathbf Q$ maps $H^{\mathfrak m}$ to itself.
Therefore,
\begin{equation}
  \label{e:direct-operator}
\mathbf P-ih\lambda-i\mathbf Q:D^{\mathfrak m}\to H^{\mathfrak m}
\end{equation}
is a holomorphic family of bounded operators.

The space $C^\infty(M;\mathcal E)$ is dense in $D^{\mathfrak m}$ as follows
from~\cite[Lemma~E.45]{DZ-Book} applied to the conjugated operator
$\widetilde{\mathbf P}$ from~\S\ref{s:conjugated-operator} (whose proof adapts directly to the case of operators
on vector bundles). Therefore, the estimates of Lemma~\ref{l:invertibility}
show that there exists a constant $C$ such that for all $\lambda\in \Omega$
\begin{equation}
\label{e:total-estimate}
\begin{aligned}
\|\mathbf u\|_{D^{\mathfrak m}}\leq C\|(\mathbf P-ih\lambda-i\mathbf Q)\mathbf u\|_{H^{\mathfrak m}}&\quad\text{for all}\quad \mathbf u\in D^{\mathfrak m},\\
\|\mathbf v\|_{D^{-\mathfrak m}}\leq C\|(\mathbf P-ih\lambda-i\mathbf Q)^*\mathbf v\|_{H^{-\mathfrak m}}&\quad\text{for all}\quad \mathbf v\in D^{-\mathfrak m}.
\end{aligned}
\end{equation}
Here $C$ and $\mathbf Q$ depend on~$h$, however we already fixed $h$ small enough above.

By a standard argument from functional analysis (see for example the proof of~\cite[Theorem~5.30]{DZ-Book}), the estimates~\eqref{e:total-estimate} imply that the operator~\eqref{e:direct-operator} is invertible for all $\lambda\in\Omega$. 
Since $\mathbf Q$ is smoothing, it is a compact operator
$D^{\mathfrak m}\to H^{\mathfrak m}$. 
It follows that $\mathbf P-ih\lambda:D^{\mathfrak m}\to H^{\mathfrak m}$
is a Fredholm operator of index~0 for all $\lambda\in\Omega$. 
Recalling that $\mathbf P=-ih\mathbf X$ and $\Omega$
is an arbitrary compact subset of
$$
\Omega_{\mathfrak m}:=\{\lambda\in\mathbb C\mid\Re\lambda>\max(r_u(m_u),r_s(m_s))\},
$$
we get the following
Fredholm property:
\begin{equation}
  \label{e:key-fredholm}
\mathbf X+\lambda:D^{\mathfrak m}\to H^{\mathfrak m}\quad\text{is a Fredholm operator
of index~0 for all }\lambda\in\Omega_{\mathfrak m}.
\end{equation}
Recall from~\eqref{e:P-R-resolvent-gen} that the Pollicott--Ruelle resolvent
$R_{\mathbf X}(\lambda)$ was defined for $\Re\lambda>C_{\mathbf X}$ by
\begin{equation}
  \label{e:prback}
R_{\mathbf X}(\lambda)\mathbf f:=\int_0^\infty e^{-\lambda t}e^{-t\mathbf X}\mathbf f\,dt\quad\text{for}\quad
\mathbf f\in C^\infty(M;\mathcal E).
\end{equation}
The operator $e^{-t\mathbf X}$ is bounded on the space
$H^{m_s}(M;\mathcal E)$ locally uniformly in~$t$. Since $e^{-t\mathbf X}$
forms a group in~$t$, we see that there exists a constant
$C_{\mathbf X}(m_s)\geq C_{\mathbf X}$ such that
$$
\|e^{-t\mathbf X}\|_{H^{m_s}\to H^{m_s}}=\mathcal O(e^{C_{\mathbf X}(m_s)t})\quad\text{as}\quad t\to\infty.
$$
For $\Re\lambda>C_{\mathbf X}(m_s)$ and $\mathbf f\in C^\infty$, the integral~\eqref{e:prback} converges in the space $H^{m_s}$ and thus (recalling~\eqref{e:anis-squeeze}) in the larger space $H^{\mathfrak m}$. Thus $\mathbf u:=R_{\mathbf X}(\lambda)\mathbf f$
lies in $H^{\mathfrak m}$ and (recalling~\eqref{e:inversor-gen}) satisfies
$(\mathbf X+\lambda)\mathbf u=\mathbf f$. 
It follows that
the range of the operator~\eqref{e:key-fredholm} contains $C^\infty(M;\mathcal E)$
and is thus dense in $H^{\mathfrak m}$.
From the Fredholm property we then see that when $\Re\lambda>\max(C_{\mathbf X}(m_s),r_u(m_u),r_s(m_s))$, the operator~\eqref{e:key-fredholm} is invertible and its inverse coincides on $C^\infty(M;\mathcal E)$ with the Pollicott--Ruelle resolvent $R_{\mathbf X}(\lambda)$.
Now by Analytic Fredholm Theory~\cite[Theorem~C.8]{DZ-Book} we see that
$$
(\mathbf X+\lambda)^{-1}:H^{\mathfrak m}\to D^{\mathfrak m},\quad
\lambda\in \Omega_{\mathfrak m}
$$
is meromorphic with poles of finite rank. This operator gives the meromorphic
continuation of the Pollicott--Ruelle resolvent, which finishes the proof of Theorem~\ref{t:main}.

\medskip\noindent\textbf{Acknowledgements.}
The author was supported by NSF CAREER grant DMS-1749858
and a Sloan Research Fellowship.
The author would like to thank an anonymous referee for many suggestions to improve the presentation.

\bibliographystyle{alpha}
\bibliography{Dyatlov,General,Ruelle}

\end{document}